\documentclass[a4paper, 11pt]{article}
\usepackage[utf8x]{inputenc}
\usepackage{setspace}
\usepackage{amsmath,amsthm,amssymb}
\usepackage{mathrsfs}
\usepackage{mdframed}
\usepackage{breakcites}
\usepackage{marvosym}
\usepackage{amssymb}
\usepackage{amsthm}
\usepackage{bbm}
\usepackage{bbold}
\usepackage{stmaryrd}
\usepackage{mathrsfs}
\usepackage{enumerate}
\usepackage{sansmath}
\usepackage{tikz}
\usepackage{graphicx}
\usepackage{boxedminipage}
\usepackage{placeins}
\usepackage{wrapfig}
\usepackage{etoolbox}
\usepackage{fancyhdr}
\usepackage{tocloft}
\usepackage[margin=40mm]{geometry}
\usepackage{palatino}
\usepackage{hyperref}

\theoremstyle{theorem}
\newtheorem{theorem}{Theorem}
\newtheorem{lemma}[theorem]{Lemma}

\newtheorem{fact}[theorem]{Fact}

\theoremstyle{definition}
\newtheorem{definition}[theorem]{Definition}

\newtheorem{remark}[theorem]{Remark}
\newtheorem{question}[theorem]{Question}
\newtheorem{axiom}[theorem]{Axiom}

\renewcommand{\leq}{\leqslant}

\usetikzlibrary{decorations.pathreplacing}
\renewenvironment{description}[1][0pt]
  {\list{}{\labelwidth=0pt \leftmargin=#1
   }}
  {\endlist}

\title{Universism and Extensions of $V$}
\author{Carolin Antos\footnote{Zukunftskolleg, University of Konstanz. E-mail: carolin.antos-kuby@uni-konstanz.de.}, Neil Barton\footnote{Fachbereich Philosophie, University of Konstanz. E-mail: neil.barton@uni-konstanz.de.}, and Sy-David Friedman\footnote{Kurt G\"{o}del Research Center. E-mail: sdf@logic.univie.ac.at.}}
\date{22 June 2020\footnote{Order of authors is alphabetical. We wish to thank Joel Hamkins, Toby Meadows, Jonas Reitz, and Kameryn Williams for insightful and helpful comments, as well as audiences in Birmingham, Cambridge, Toulouse, and Vienna for the opportunity to present the ideas and subsequent discussion. Special mention must be made of Marcus Giaquinto, Bob Hale, Alex Paseau, Ian Rumfitt, and Chris Scambler, whose help during the preparation and examination of the second author's PhD thesis (where versions of this material appeared) were absolutely invaluable. We wish also to thank three anonymous reviewers from \textit{The Review of Symbolic Logic}, whose careful reading and commentary helped to improve the paper. The first author was supported by the Zukunftskolleg, the European Commission (Marie Curie Individual Fellowship, Proposal number 752211), and the VolswagenStiftung (Project: \textit{Forcing: Conceptual Change in the Foundations of Mathematics}). The second author was supported by the UK Arts and Humanities Research Council, the FWF (Austrian Science Fund, Project P 28420), and VolswagenStiftung (Project: \textit{Forcing: Conceptual Change in the Foundations of Mathematics}). The third author was supported by the FWF (Austrian Science Fund, Project P 28420).}}

\begin{document}

\maketitle

\begin{abstract}

A central area of current philosophical debate in the foundations of mathematics concerns whether or not there is a single, maximal, universe of set theory. Universists maintain that there is such a universe, while Multiversists argue that there are many universes, no one of which is ontologically privileged. Often model-theoretic constructions that add sets to models are cited as evidence in favour of the latter. This paper informs this debate by developing a way for a Universist to interpret talk that seems to necessitate the addition of sets to $V$. We argue that, despite the prima facie incoherence of such talk for the Universist, she nonetheless has reason to try and provide interpretation of this discourse. We present a method of interpreting extension-talk ($V$-logic), and show how it captures satisfaction in `ideal' outer models and relates to impredicative class theories. We provide some reasons to regard the technique as philosophically virtuous, and argue that it opens new doors to philosophical and mathematical discussions for the Universist.
\end{abstract}

\section*{Introduction}

Recent discussions of the philosophy of set theory have often focussed whether there is a maximal set-theoretic universe or structure for our discourse about sets.\footnote{Obviously, because of the set-theoretic paradoxes, we are using the terms `universe' and `structure' informally here, rather than as understood through their set-theoretic definitions. As we'll see, different philosophical views will understand this term differently. The unified core among all views seems to be that a universe, in this philosophical sense, provides a domain for our first-order set-theoretic quantifiers.} We begin with two views:

\begin{description}
\item[Universism.] There is a proper-class-sized universe (denoted by `$V$'), which is (a) unique, (b) cannot be extended, and (c) contains all the sets.
\end{description}

However, some have seen set-theoretic results (such as forcing and ultrapowers) as providing evidence that the subject matter of set theory is not constituted by some single  `absolute' universe, but rather a plurality of different universes.\footnote{See here, for example, \cite{Hamkins2012a} or \cite{Steel2014a}.} Universism thus contrasts sharply with the following view:

\begin{description}
\item[Multiversism.] The subject matter of set theory is constituted by a plurality of different universes, and no single one of them contains all\footnote{Of course, the Multiversist is likely to run into the usual problems of expressibility and generality relativism here with their use of the quantifier `all'. There are options available, however---see \cite{RayoUzquiano2006a} for an excellent collection of essays on the topic.} the sets.
\end{description}

This question of ontology can be linked to the following idea: Do set theorists merely study models of set theory (similar to a group theorist's study of groups), or can set theory be understood as about a single domain of sets? The former idea is naturally motivated by Multiversism, and the latter by Universism. The current paper concerns how a Universist might bring a version of the algebraic Multiverse-inspired perspective to bear on her universe of sets.

There are several ways of fleshing out Multiversism. For instance, one might hold that the intended universes contain all the ordinals, but that subsets can be added to any given universe.\footnote{See \cite{Steel2014a} and \cite{Meadows2015a}.} Alternatively, one might hold that there is a definite powerset operation, but any universe can be extended to another with more ordinals.\footnote{See here \cite{Rumfitt2015a}, \cite{Isaacson2011a}, and \cite{Hellman1989a}, with historical roots in \cite{Zermelo1930a}.} Still further, one might take a universe to be extendable both by adding more ordinals and/or by adding subsets.\footnote{See \cite{ArrigoniFriedman2013a}.} Some views advocate that in addition to extension by adding subsets and ordinals, every intended universe is non-well-founded from the perspective of another.\footnote{See \cite{Hamkins2012a}.}

The present paper is concerned with how we might interpret \textit{width} extensions of universes (i.e. the addition of subsets but not ordinals) for the Universist, clarifying the dialectic with respect to certain constructions in set theory (which we discuss later), and showing how she can use resources one might be tempted to think are prohibited. We will therefore contrast Universism with the following species of Multiversism:

\begin{description}
\item[Width Multiversism.] Letting $\mathcal{V}$ be a variable ranging over universes, for any particular set-theoretic universe $\mathcal{V}$, there is another universe $\mathcal{V}'$ such that:

\begin{enumerate}[(i)]
\item $\mathcal{V}'$ is an end extension of $\mathcal{V}$.\footnote{$\mathcal{V}'$ is an \textit{end-extension} of $\mathcal{V}$ iff whenever $\mathcal{V}'$ thinks that $a \in b$ (for $b \in \mathcal{V}$) then $\mathcal{V}$ thinks that $a \in b$. In other words, $\mathcal{V}'$ cannot see new elements of sets in $\mathcal{V}$ that $\mathcal{V}$ was oblivious to. This version of Width Multiversism is formulated so as to not depend on an independent notion of transitivity, and so could work on pictures such as \cite{Hamkins2012a}. If one believes that there is an absolute notion (philosophically speaking) of transitivity, then one might wish to replace this condition with the stronger claim that both $\mathcal{V}$ and $\mathcal{V}'$ are transitive. Nothing hangs on this modification for the points we wish to make.}
\item $\mathcal{V}$ is a proper subclass of $\mathcal{V}'$.
\item $\mathcal{V}$ and $\mathcal{V}'$ have the same ordinals.
\end{enumerate}
\end{description}

One critique of the Universist's position from width multiversists concerns how they are able to interpret talk of width extensions (i.e. constructions that add subsets, but not ordinals, to some $\mathcal{V}$).\footnote{See here \cite{Hamkins2012a}, which we discuss in detail later.} As we shall argue in detail later, certain axioms appear to require higher-order relations between universes and their extensions in order to be formulated properly. Just to briefly foreshadow, we shall discuss the Inner Model Hypothesis\footnote{See \cite{Friedman2006a}.}, the Axiom of $\sharp$-Generation\footnote{See \cite{FriedmanHonzik2016a}.}, and certain axioms coming from set-theoretic geology\footnote{See \cite{FuchsHamkinsReitz2015a}.}. This paper is thus addressed towards the following pair of questions:

\begin{enumerate}[(1.)]
\item Can the Universist provide an interpretation of the discourse required to formulate these extension-requiring axioms as concerned with $V$? 
\item How \textit{natural} is the interpretation they can provide?
\end{enumerate}

\noindent We shall argue for the following claims:

\begin{enumerate}[(I)]
\item \textit{If} one holds that the Universist should be able to interpret set-theoretic discourse \textit{naturally} (in a sense we make precise later), \textit{then} there is pressure to provide an interpretation of axioms that \textit{seem} concerned with higher-order relations between $V$ and extensions thereof. 
\item There is a logical system ($V$-logic), formalisable with the use of resources beyond first-order, that suffices to provide an interpretation of these axioms. 
\item By arguing for the use of higher-order set theory, and the existence of countable transitive models elementarily equivalent to $V$ for portions of its higher-order theory, the Universist can also provide a \textit{natural} interpretation of these axioms. 
\end{enumerate}

The main goal is thus philosophical; to argue that the Universist has a well-motivated way of interpreting axioms that seem to need extensions of $V$. However, along the way we will articulate several mathematical theories and results that are useful in establishing this conclusion. Some readers may only be interested in the mathematical details, and these can be understood solely from reading sections 4,5, and 6. Other readers may not be interested in the mathematical details, in which case the summary just before the start of section 6 can be treated as a black box, and sections 4 and 5 skipped. In more detail, here's the plan:

 After these introductory remarks, we discuss (\S1) a challenge for the Universist; to make sense of talk concerning extensions of $V$. We discuss a constraint on interpreting extensions that motivates Multiversist criticisms---that the interpretation provided preserve as much naive thinking as possible. We then (\S2) discuss a way in which a Universist might miss insight from extensions of $V$: One can formulate axioms using extensions that relate to $V$. We mention some existing interpretations, in particular the use of countable transitive models. Next (\S3) we identify a way this strategy (as stated) might be viewed as \textit{limiting}; it fails to account for axioms relating to higher-order properties. We provide some examples of these; Inner Model Hypotheses, $\sharp$-generation, and certain axioms coming from set-theoretic geology. The next sections (\S4, \S5, and \S6) provide the elements of a positive response to the problem. \S4 provides an exposition of a logical system ($V$-logic) that suffices to interpret extensions. \S5 then shows how $V$-logic can be coded within class theory. \S6 then discusses how we can use the mathematical results of \S4 and \S5 to reduce $V$-logic to the countable, and argues that this interpretation is philosophically natural. Finally (\S7) we identify some open questions and conclude that new avenues of philosophical and mathematical research are opened by this interpretation of extensions.

\section{Extensions and insight}

Before we start getting into the details, some philosophical set up is necessary to elucidate the problem and some constraints on a solution. In this section we do this by explaining some commitments a Universist might have, some challenges raised by criticisms of Universism and the constraints they suggest, and the rough shape of the problem. Specific examples will then be given more detailed consideration in \S2.

One salient feature of the Universist's position is that often they can find interpretations of statements involving extensions.\footnote{In certain cases they may need to use large cardinals, but we suppress this detail for the moment.} A good example here is with \textit{set forcing}, which adds a subset $G$ to a universe $\mathcal{V}$ (for nontrivial forcing). From the Universist's perspective there are a number of options for interpreting occurrences of terms like `$V[G]$' in the practice of set theory. They could, for example, interpret this via the use of the forcing relation or taking `$V$' to denote some countable transitive model of the required form.\footnote{For some philosophical discussion of these options regarding forcing see \cite{Barton2019b}.} There is thus a question of if and/or why the Universist should be perturbed by the use of extensions in set-theoretic practice. Width Multiversists are well-aware of this strategy and often supplement the mere existence of model-construction methods with appeals to philosophical desiderata. Hamkins is a somewhat representative example here:

\begin{quote}
A stubborn geometer might insist---like an exotic-travelogue writer who never actually ventures west of seventh avenue---that only Euclidean geometry is real and that all the various non-Euclidean geometries are merely curious simulations within it. Such a position is self-consistent, although stifling, for it appears to miss out on the geometrical insights that can arise from the other modes of reasoning. Similarly, a set theorist with the universe view can insist on an absolute background universe $V$, regarding all forcing extensions and other models as curious complex simulations within it. (I have personally witnessed the necessary contortions for class forcing.) Such a perspective may be entirely self-consistent, and I am not arguing that the universe view is incoherent, but rather, my point is that if one regards all outer models of the universe as merely simulated inside it via complex formalisms, one may miss out on insights that could arise from the simpler philosophical attitude taking them as fully real. (\cite{Hamkins2012a}, p426)
\end{quote}

 Central to Hamkins' point is that often set-theorists will examine extensions of universes whilst using the symbol `$V$' to denote the ground model. `$V$' and terms for sets in the extension cannot {\it literally} denote $V$ and sets in an extension of $V$ for the Universist, as in that case there would be sets outside $V$ (by the design of the extending construction), and hence her position would be false. Thus, she has to reinterpret use of the relevant theoretical terms so as not to commit herself to sets outside $V$ (for example, in the case of forcing extensions, through the use of the forcing relation, countable transitive models, or Boolean ultrapower map). Hamkins' point is that the kinds of interpretation given (in his words ``curious complex simulations'') obscure possible mathematical insight. While we remain agnostic on to what degree the Universist should be troubled by Hamkins' remarks, it is nonetheless a worthwhile philosophical and mathematical endeavour (and an important motivator for this paper) to examine ways in which the Universist might counteract this potential loss of mathematical `insight'. A few remarks are thus in order to clarify what this might come down to.

One important point that we will spell out in detail later is that there are axioms that \textit{appear} to require the use of extensions of $V$ in their formulation, and this is the problem that forms the focus of this paper. We will examine interpretations that legitimise the use of such talk on the Universist's picture. The position is \textit{somewhat} similar to the one Hilbert took concerning infinite sets (under one interpretation of a period of Hilbert's thought)---infinite sets for him are fictitious but discourse involving them might be useful insofar as it facilitates an understanding of finitary objects. Similarly for the Universist, axioms involving extensions are useful insofar as they allow us to make claims concerning $V$, even if the extensions do not strictly speaking exist. We therefore pose the following:

\begin{description}
\item[Hilbertian Challenge.]\footnote{We use the name `Hilbertian Challenge' to echo Hilbert's desire to show that transfinite resources would never lead to finitary contradiction. Of course, by G\"{o}del's Second Incompleteness Theorem, Hilbert's Programme in its original form was doomed from the outset (though for a heroic defence of Hilbert's aims see \cite{Detlefsen1986a}, \cite{Detlefsen1990a}, and \cite{Detlefsen2001a}). In the present context then, we want to provide philosophical and mathematical reasons to accept the use of extension-talk in formulating axioms about $V$. There is a substantial and interesting question as to {\it how} Hilbertian this challenge really is. We certainly do not claim that it accords with all of Hilbert's writing. Indeed, one might take Hilbert as requiring {\it conservativity} rather than merely a lack of contradictions. All we wish to identify here is that there are certain parallels between Hilbertian Finitism and Universism. Thanks to Giorgio Venturi for emphasising this point.}  Provide philosophical reasons to legitimise the use of extra-$V$ resources for formulating axioms concerning $V$.
\end{description}

Before we embark on the details, a salient question here is exactly how the Hilbertian Challenge should be answered. It is one thing to provide an interpretation of extensions, and another for that interpretation to be {\it good} or \textit{natural} facilitating the `insight' Hamkins talks about. In order to examine the strongest possible challenge to Universism, we therefore impose the following additional constraint: 

\begin{description}
 \item[The Methodological Constraint.] In responding to the Hilbertian Challenge, do so in a way that accounts for as much as possible of our naive thinking about extensions and links it to structural features of $V$. In particular, if we wish to apply an extending construction to $V$, there should be an {\it actual} set-theoretic model, resembling $V$ as much as possible, that has an extension similar to the one we would like $V$ to have.\footnote{A similar constraint is examined solely with respect to forcing constructions in \cite{Barton2019b}, where some conditions on the `naturalness' of interpretations of forcing are discussed.}
\end{description}

 Immediately there is a problem: the most `naive' way of interpreting the talk would simply be to interpret `$V$' as denoting $V$, and `ideal' sets as denoting actual sets outside $V$, which is obviously problematic. However,  we should keep in mind that the Universist only wishes to make claims (to be discussed later) about sets \textit{in $V$} using talk that appears to involve resources external to $V$. She is thus not interested in objects that would have to be outside $V$ were they to exist. Rather her interest is in the objects of $V$, and how our thinking relates to {\it them}. Thus, the challenge is to see {\it how much} of our naive thinking can be interpreted on her picture, without {\it actually} committing to the existence of sets outside $V$.

In \cite{Barton2019b}, the second author argued that many insights relating to \textit{forcing} could be captured through the use of a countable transitive model elementarily equivalent to $V$. Here we will extend this line of thought, arguing that a Universist can capture far more than one might think (even statements concerning certain higher-order axioms) and that this paves the way for new and interesting philosophical and mathematical discussions from Universist points of view. In particular, we shall argue that we can capture more than just {\it forcing} extensions, but can in fact naturally interpret talk concerning {\it arbitrary} extensions of $V$ on a Universist picture. Indeed, we shall show that using some impredicative class theory, we can capture satisfaction in arbitrary outer models in which $V$ appears standard. This provides a definition of satisfaction in both forcing extensions of various kinds and also non-forcing extensions (such as if $V$ resulted from a generating sharp). Reduction of this situation to the countable (again, possibly using impredicative class theory) then facilitates a philosophically virtuous way of interpreting extension-talk for a Universist. 

\section{Formulating axioms with extensions}

With the philosophical backdrop in place, we will now state the rough abstract shape of the problem we shall consider. In formulating axioms, we shall see that the use of `ideal' objects outside $V$ can lead to triviality or apparent falsehood, even when we are trying to make a claim about objects within $V$. Before we dig into the details, a toy example is helpful to understand how consideration of extensions might be useful for making a claim about $V$. Let $\Phi$ and $\Xi$ be conditions on universes. Suppose we wish to assert the following principle about $V$:

\begin{description}
 \item[\textbf{Principle-$\Xi^V$.}] If there is a proper width extension of $V$ such that $\Phi$, then $\Xi$ holds of $V$.
\end{description}

The problem with Principle-$\Xi^V$ is that, given a naive interpretation, it will always come out as true, but fail to capture the intended aspect of $V$ (namely that $\Xi$ holds of $V$). For the antecedent (on its natural reading) is trivially false, and so the conditional is true. But this provides us with no reason to think that $\Xi$ is actually true of $V$ which, presumably, was the intended consequence of asserting the putative axiom in the first place.

Some axioms can indeed be formulated in this way. Consider the following case. {\it Martin's Axiom} is a well-known proposed axiom, and is normally formulated as follows:\footnote{One can also find a use of this analogy in \cite{BartonFriedman2017a}.}

 \begin{definition}
  ($\mathbf{ZFC}$) {\it Martin's Axiom.} Let $\kappa$ be a cardinal such that $\kappa < |\mathcal{P}(\omega)|$. $\mathsf{MA}(\kappa)$ is the claim that for any partial order $\mathbb{P}$ in which all maximal antichains are countable (i.e. $\mathbb{P}$ has the countable chain condition), and any family $\mathcal{D}$ of dense sets of $\mathbb{P}$ such that $|\mathcal{D}| \leq \kappa$, there is a filter $F$ on $\mathbb{P}$ such that for every $D \in \mathcal{D}$, $F \cap D \not = \emptyset$. {\it Martin's Axiom} (or $\mathsf{MA}$) is then the claim that $\forall \kappa < |\mathcal{P}(\omega)|$, $\mathsf{MA}(\kappa)$.
 \end{definition}

Effectively, Martin's Axiom rendered in this form states that the universe has already been saturated under forcing of a certain kind.\footnote{The same goes for other similar forcing axioms.} However, we could equivalently formulate Martin's Axiom as the following {\it absoluteness} principle:

 \begin{definition} ($\mathbf{ZFC}$) \cite{Bagaria1997a} {\it Absolute-$\mathsf{MA}$.} We say that $\mathfrak{M}$ satisfies {\it Absolute-$\mathsf{MA}$} iff whenever $\mathfrak{M}[G]$ is a generic extension of $\mathfrak{M}$ by a partial order $\mathbb{P}$ with the countable chain condition in $\mathfrak{M}$, and $\phi(x)$ is a $\Sigma_1(\mathcal{P}(\omega_1))$ formula (i.e. a first-order formula containing only parameters from $\mathcal{P}(\omega_1)$), if $\mathfrak{M}[G] \models \exists x \phi(x)$ then there is a $y$ in $\mathfrak{M}$ such that $\phi(y)$.
 \end{definition}

This version of Martin's Axiom is interesting given our current focus; it asserts that if something is true in an extension of a particular kind, then it already holds in $\mathfrak{M}$. In this way, it conforms to the natural idea for the Universist that if a set of a certain kind is \textit{possible} then it is (in some sense) \textit{actual}.\footnote{A full examination of this line of thinking is outside the scope of the current paper, but further discussion of the idea is available in \cite{Bagaria2005a}, \cite{Bagaria2008a}, and \cite{BartonSa}.} Suppose then that the Universist is only aware of Absolute-$\mathsf{MA}$ and not Martin's Axiom as usually stated. Supposing that she viewed it as a maximality principle worthy of study, could she meaningfully analyse the axiom for its truth or falsity in $V$ despite its apparent reference to extensions?

The answer is clearly ``Yes!''. This is because (as will be familiar to specialists) despite the fact that the Universist does not countenance the literal existence of the extensions, she can nonetheless capture the notion of {\it satisfaction in a particular set-generic forcing extension} using formulas that are first-order definable over $V$. More specifically, given a formula $\phi$ (or set of formulas of some bounded complexity), the Universist can define a class of $\mathbb{P}$-names in $V$, and a relation $\Vdash_\mathbb{P}$ (known as a {\it forcing relation}) such that: For $p \in \mathbb{P}$, if $p$ were in some (`ideal', `non-existent') $\mathbb{P}$-generic $G$, and $p \Vdash_\mathbb{P} \phi$ holds in $V$, then $V[G]$ would have to satisfy $\phi$ were it to exist. Moreover, if some `ideal' $V[G]$ were to satisfy $\phi$, then there is a $q \in G \subseteq \mathbb{P}$ such that $q \Vdash_{\mathbb{P}} \phi$.\footnote{See \cite{Kunen2013a}, Ch. IV (esp. \S IV.5.2) for details.} In this way, her $V$ has access to the satisfaction relation of `ideal' outer models. To be clear, from the Universist perspective, all she is really doing here is schematically talking about the various relations $\Vdash_\mathbb{P}$ for the relevant $\phi$, and various $q \in \mathbb{P}$ in her model, it just so happens that this talk of $\Vdash_\mathbb{P}$ mimics what would be true in extensions of $V$ (were they to exist). The Universist can then reformulate Absolute-$\mathsf{MA}$ as follows:

\begin{definition}
 ($\mathbf{ZFC}$) {\it Absolute-$\mathsf{MA}^{\Vdash}$.} We say that $V$ satisfies {\it Absolute-$\mathsf{MA}^{\Vdash}$} iff whenever $\mathbb{P} \in V$ is a partial order with the countable chain condition in $V$, and $\phi(x)$ is a $\Sigma_1(\mathcal{P}(\omega_1))$ formula, if there is a $p \in \mathbb{P}$ and $\Vdash_\mathbb{P}$, such that $p \Vdash_\mathbb{P} \exists x \phi(x)$, then there is a $y$ in $V$ such that $\phi(y)$.
\end{definition}

Thus, by coding {\it satisfaction} in outer models (without admitting their existence), the Universist can express the content of Absolute-$\mathsf{MA}$ through Absolute-$\mathsf{MA}^{\Vdash}$.\footnote{This will, of course, be equivalent to the standard version of Martin's Axiom. See \cite{Bagaria1997a} for details.} Now, the use of such a forcing relation is very syntactic, and so it is unclear how the Methodological Constraint is satisfied---there is no model very similar to $V$ being extended. However, it is well-known that if one moves to a countable transitive model $\mathfrak{M}$ of the required form, many extensions of $\mathfrak{M}$ are readily available, and some of these models \textit{can} bear a close resemblance to $V$. It is then false (as Hamkins acknowledges) to say that the Universist {\it cannot} provide an interpretation on which the set theorist's statements come out as true. It is simply that, given many of the interpretations, `$V$' does not denote $V$. Koellner expresses the point as follows:

\begin{quote}
Hamkins' aim is to ``legitimize the actual practice of forcing, as it is used by set theorists''... The background is this: Set theorists often use `$V$' instead of `$M$' and so write `$V[G]$'. But if $V$ is the entire universe of sets then $V[G]$ is an ``illusion''. What are we to make of this? Most set theorists would say that it is just an abuse of notation. When one is proving an independence result and one invokes a transitive model $M$ of $\mathbf{ZFC}$ to form $M[G]$ one wants to underscore the fact that $M$ could have been any transitive model of $\mathbf{ZFC}$ and to signal that it is convenient to express the universality using a special symbol. The special symbol chosen is `$V$'. This symbol thus has a dual use in set theory---it is used to denote the universe of sets and (in a given context) it is used as a free-variable to denote any countable transitive model (of the relevant background theory). (\cite{Koellner2013a}, p19)
\end{quote}

Koellner's point here is that the Universist can interpret extension constructions (in particular those concerned with independence), by moving to a countable transitive model (where extensions are uncontroversially available) and conducting the construction there. To transfer the strategy to the current case, we could argue that axioms using extension talk concerning $V$ could be formulated about a countable transitive model $\mathfrak{M}$ satisfying $\mathbf{ZFC}$, and then whatever is proved about $\mathfrak{M}$ merely on the basis of its $\mathbf{ZFC}$-satisfaction can be exported back to $V$.

We defer a detailed consideration of extant interpretations of certain kinds of extension-talk (e.g. forcing) to different work.\footnote{Hamkins himself has several criticisms of these techniques (including the countable transitive model strategy) which the second author addresses in \cite{Barton2019b}. For examination of axioms using countable transitive models elementarily equivalent to $V$ (or some $\mathcal{V}$), see also \cite{ArrigoniFriedman2013a} and \cite{AntosFriedmanHonzikTernullo2015a}.} Nonetheless, a review of the challenges faced and a standard response will be helpful for seeing the main aims and strategy of the current paper. 

We start by noting the following weakness in the countable transitive model strategy (as stated) for the purposes of interpreting axioms concerning $V$ that make mention of extensions: There is no guarantee that $\mathfrak{M}$ and $V$ satisfy the same sentences (though they agree on $\mathbf{ZFC}$). This is of central importance in the current context, part of the point of asserting new axioms for the Universist is to reduce independence concerning $V$, and so we require (according to the Methodological Constraint) as much similarity between $V$ and a countable transitive model as possible.

The situation can be somewhat remedied by having a countable transitive model that satisfies {\it exactly the same} parameter-free first-order sentences as $V$. We will denote such an object by $\mathfrak{V}$. Later (\S6) we will discuss possible reasons why a Universist might accept the existence of such an entity, however for now let us assume that such a $\mathfrak{V}$ exists to see the rough shape of a response to the Hilbertian Challenge. On this picture, when a theorist uses the term `extension of $V$', we interpret them as talking about some $\mathfrak{V}$ (which is elementarily equivalent to $V$) and its relationship to $V$. We have a model, \textit{very similar to $V$} that \textit{really is} being extended. Any result in the first-order language of $\mathbf{ZFC}$ that holds in $\mathfrak{V}$ can then be exported back to $V$ via the elementary equivalence. Moreover, we have an explanation of how talk of forcing over $\mathfrak{V}$ is related to forcing in the context of $V$, both $\mathfrak{V}$ and $V$ will have forcing relations for the relevant formulas and partial orders, it is just that in the case of $\mathfrak{V}$ these relations correspond to the \textit{actual} existence of generics. Use of a countable transitive model elementarily equivalent to $V$ is thus a viable alternative for interpreting extension talk in line with the Hilbertian Challenge and Methodological Constraint, at least as far as parameter-free first-order truth goes. Whilst there are other options for how one might interpret extensions of $V$ (such as via use of a forcing relation or Boolean-valued models, or considering substructures of $V$ obtained by a Boolean ultrapower map),  for now we will simply take it as given that the countable transitive model strategy works well for parameter-free first-order truth:\footnote{These options, and why we might regard the countable transitive model strategy as privileged, are discussed by the second author in \cite{Barton2019b}.} We have a model very similar to $V$ actually being extended and explained how it is linked to truth in $V$. This is the rough form of problem and response we shall articulate in the rest of the paper, but with respect to \textit{arbitrary} extensions.

The core difficulty, to be explained in the next section and overcome in the sections thereafter, is that we might still regard this version of the countable transitive model strategy as somewhat limited. In particular, there are axioms that we can state using higher-order resources \textit{in combination} with extensions, for which such a $\mathfrak{V}$ might not be appropriately linked to truth in $V$. If we want to allow for interpretation of these axioms, we require a modification of the countable transitive model strategy. The next section (\S3) provides an exposition of these axioms, before we show how to modify the countable transitive model strategy to accommodate them (\S\S4,5,6).

\section{Higher-order axioms}

This section provides exposition of several axioms that possibly require more for their satisfactory formulation than what we can guarantee through the existence of a countable transitive model elementarily equivalent to $V$. In particular, as we shall see, the issue concerns the resemblance between the relevant countable model and {\it higher-order} properties.

\subsection{Inner model richness}

Extensions of $V$ are useful for postulating the existence of many inner models. The {\it Inner Model Hypothesis} does exactly this, using extensions of a model $\mathfrak{M}$ in order to make claims about the inner models of $\mathfrak{M}$:

 \begin{definition}
(Informal\footnote{We will show how to code this axiom formally in \S\S4--6.}) \cite{Friedman2006a} $\mathfrak{M}$ (a model of $\mathbf{ZFC}$) satisfies the {\it $\mathfrak{M}$-Inner Model Hypothesis} (henceforth `$\mathfrak{M}$-$\mathsf{IMH}$') iff  whenever $\phi$ is a parameter-free first order sentence that holds in an inner model $I^{\mathfrak{M}^{*}}$ of an outer model $\mathfrak{M}^*$ of $\mathfrak{M}$, there is an inner model $I^\mathfrak{M}$ of $\mathfrak{M}$ that also satisfies $\phi$. The \textit{Inner Model Hypothesis} (henceforth `$\mathsf{IMH}$') is the $V$-Inner Model Hypothesis.
\end{definition}

\begin{remark}
As we shall see (and this is one of the foci of this paper) the content of this claim depends on the interpretation we give to the notion of `outer model of $V$' and `inner model of $V$'. Later, we will show that admitting certain impredicative resources facilitates the coding of the $\mathsf{IMH}$, thus for now we leave it informally formulated. The reader experiencing metaphysical queasiness can, for the moment, think about the $\mathfrak{M}$-$\mathsf{IMH}$ as formulated about an arbitrary countable transitive model $\mathfrak{M}$ of $\mathbf{ZFC}$.
\end{remark}

The $\mathfrak{M}$-$\mathsf{IMH}$ thus states that $\mathfrak{M}$ has many inner models, in the sense that any sentence $\phi$ true in an inner model of an outer model of $\mathfrak{M}$ is already true in an inner model of $\mathfrak{M}$. In this way, $\mathfrak{M}$ has been maximised with respect to {\it internal consistency} (see Figure 1 for visual representation of an application of the $\mathfrak{M}$-$\mathsf{IMH}$).

\begin{figure}
\begin{center}
\caption{A visual representation of an application of the $\mathfrak{M}$-$\mathsf{IMH}$}
\begin{tikzpicture}
\node (a) at (-5,8.2) {$\mathfrak{M}^*$};
\node (b) at (5,8.2) {$\mathfrak{M}^*$};
\node (c) at (-4,8.2) {$I^{\mathfrak{M}^*}$};
\node (d) at (4, 8.2) {$I^{\mathfrak{M}^*}$};
\node (e) at (-3, 7) {$\phi$};
\node (f) at (0, 7) {$\phi$};
\node (g) at (-3, 8.2) {$\mathfrak{M}$};
\node (h) at (3, 8.2) {$\mathfrak{M}$};
\node (i) at (-2, 8.2) {$I^\mathfrak{M}$};
\node (j) at (2, 8.2) {$I^\mathfrak{M}$};
\draw[->, double] (e) to (f);
\draw (0,0) -- (-4,8);
\draw (0,0) -- (4,8);
\draw (0,0) -- (5,8);
\draw (0,0) -- (-5,8);
\draw (0,0) -- (-3,8);
\draw (0,0) -- (3,8);
\draw (0,0) -- (-2,8);
\draw (0,0) -- (2,8);
\end{tikzpicture}
\end{center}
\end{figure}

There are a number of reasons to find the $\mathfrak{M}$-$\mathsf{IMH}$ interesting, not least because it maximises the satisfaction of consistent sentences within structures internal to some $\mathfrak{M}$. If we could formulate the $\mathsf{IMH}$ as about $V$, it would thus be foundationally significant: The $\mathsf{IMH}$ gives us an inner model for any sentence model-theoretically compatible with the initial structure of $V$, and thus serves to ensure the existence of well-founded, proper-class-sized structures in which we can do mathematics. It also has a similarity to Absolute-$\mathsf{MA}$; the $\mathsf{IMH}$ is just more general in that it permits the consideration of \textit{arbitrary} extensions instead of merely set forcing extensions. The $\mathfrak{M}$-$\mathsf{IMH}$ also has substantial large cardinal strength; it implies the existence of inner models (of $\mathfrak{M}$) with measurable cardinals of arbitrarily large Mitchell order, and is consistent relative to the existence of a Woodin cardinal with an inaccessible above.\footnote{See \cite{FriedmanWelchWoodin2008a} for details.} However, it is also interesting in that it has various {\it anti}-large cardinal properties, the $\mathfrak{M}$-$\mathsf{IMH}$ implies that there are no inaccessibles in $\mathfrak{M}$.\footnote{See \cite{Friedman2006a}, p. 597 for details.} 

Whence the problem then for the Universist? If the Universist wishes to use the $\mathsf{IMH}$ as a new axiom about $V$, she has to examine issues concerning extensions of $V$. If they ascribe {\it no} meaning to claims concerning extensions, then the $\mathsf{IMH}$ is utterly trivial (when viewed as a principle true of the universe, rather than some substructure thereof). Under this analysis, everything true in an inner model of an outer model of $V$ is also true in an inner model of $V$, as either (i) the outer model is proper, does not exist, and hence nothing is true in an inner model of that proper outer model of $V$, or (ii) the outer model is $V$ itself, and obviously anything true in an inner model of $V$ is true in an inner model of $V$. Thus, if the Universist ascribes no meaning to the term `outer model of $V$' the inner model hypothesis fails to capture what it was designed to state (i.e. a richness of inner models).

However, even supposing that the Universist allows {\it some} interpretation of extension talk, the content that the $\mathsf{IMH}$ has is going to vary according to the resources one allows. We begin by making the following definitions:

\begin{definition}
($\mathbf{NBG}$) Let $(V, \mathcal{C}^V)$ denote the structure (satisfying $\mathbf{NBG}$) composed of $V$ with all its classes\footnote{We discuss how to interpret classes over $V$ later; see \S5.2.}.  The {\it Set-Generic Inner Model Hypothesis} is the claim that if a (first-order, parameter free) sentence $\phi$ holds in an inner model of a set forcing extension $(V[G], \mathcal{C}^V[G])$ of $(V, \mathcal{C}^V)$ (where $V[G]$ consists of the interpretations of set-names in $V$ using $G$, and $\mathcal{C}^V[G]$ consists of the interpretations of class-names in $\mathcal{C}^V$ using $G$), then $\phi$ holds in an inner model of $V$. If we insist that $\mathcal{C}^V$ consists only of $V$-definable classes, and that the inner models of $V$ and $V[G]$ are definable with set parameters, then we shall call this principle the {\it Definable Set-Generic Inner Model Hypothesis}.
\end{definition}

\begin{definition}
($\mathbf{NBG}$) Again, let $(V, \mathcal{C}^V)$ be the $\mathbf{NBG}$ structure consisting of $V$ with all its classes. The {\it Class-Generic Inner Model Hypothesis} is the claim that if a (first-order, parameter free) sentence $\phi$ holds in an inner model of a tame class forcing extension $(V[G], \mathcal{C}^V[G])$ of $(V, \mathcal{C}^V)$ (where $V[G]$ and $\mathcal{C}^V[G]$ are defined as above), then $\phi$ holds in an inner model of $V$.
\end{definition}

\begin{remark}
Before we begin teasing apart the two possible axioms, we make the following observation regarding the expressibility of these principles.\footnote{We thank an anonymous reviewer for suggesting that we examine this question.} The Definable Set-Generic Inner Model Hypothesis is expressible in $V$ as a scheme of assertions in first-order $\mathbf{ZFC}$, since we can first-order quantify over set forcings, and talk definably about their forcing relations, and using these forcing relations speak about definable inner models of the set-forcing extensions. For the full Set-Generic and Class-Generic Inner Model Hypotheses, we need to use class theory; in the first case to quantify over inner models (that may not be definable), and in the second case to quantify over the forcings in question. Note that for class forcings, one can make the inner model of the outer model definable using Jensen coding, and so the assumption on the definability of the inner models of $V$ and $V[G]$ would result in equivalent principles. Later (in the omniscience subsection of \S6), we will see that given some (non-first order) conditions on the universe, the full $\mathsf{IMH}$ can be expressible in $\mathbf{NBG}$. In any case, the results of \S6 will show that a variant of $\mathbf{MK}$ (in fact a variant of $\mathbf{NBG}$+$\Sigma^1_1$-Comprehension) also suffices to formulate all versions of the $\mathsf{IMH}$. For the time being then, we will restrict to the versions of the $\mathsf{IMH}$ that we can express in $\mathbf{NBG}$, i.e. the Set-Generic Inner Model Hypothesis and the Class-Generic Inner Model Hypothesis.
\end{remark}

We then have the following simple fact for the Set-Generic Inner Model Hypothesis:

 \begin{theorem}
($\mathbf{NBG}$)  If the Set-Generic Inner Model Hypothesis holds then $V\not = L$.
 \end{theorem}

 \begin{proof}
  Assume $V=L$ and the Set-Generic Inner Model Hypothesis. Then there is an inner model of an outer model in which $V=L$ is false (the addition of a single Cohen real $x$ over $L$ to $L[x]$ will suffice, with the relevant inner model simply being the forcing extension $L[x]$). By the Generic Inner Model Hypothesis there is an inner model of $L$ in which $V \not = L$. But $L$ is the smallest inner model, and so $V=L$ and $V \not = L$, $\bot$. 
 \end{proof}

However, though this restricted version of the $\mathsf{IMH}$ is sufficient to get us a certain richness of inner models (enough to break $V=L$) we get more if we allow class forcings. This is brought out in the following:

\begin{theorem}
($\mathbf{NBG}$) Assuming the consistency of the existence of a $V_\kappa$ elementary in $V$, there is a model satisfying the Set-Generic Inner Model Hypothesis that violates the Class-Generic Inner Model Hypothesis.
 \end{theorem}
\begin{proof}
Take a model $\mathfrak{M}$ of $V=L$ containing a $V^\mathfrak{M}_\kappa = L^\mathfrak{M}_\kappa \prec V^\mathfrak{M} = L^\mathfrak{M}$. We work from the perspective of $\mathfrak{M}$. For any particular $\beta$, let $Col(\omega, \beta)$ be the L\'{e}vy collapse of $\beta$ to $\omega$, and let $G$ be generic for $Col(\omega, \kappa)$. We claim that $(L[G], Def^L[G])$ satisfies the Set-Generic Inner Model Hypothesis. It suffices to show that if a sentence holds after forcing with $Col(\omega, \lambda)$ for some $\lambda$, then this $\lambda$ can be chosen to be less than $\kappa$. This is because any set-forcing can be absorbed into $Col(\omega, \lambda)$ for some $\lambda$,\footnote{See here \cite{Cummings2010a} (\S14) for technical explanation, and \cite{Barton2019b} for further philosophical remarks. We thank Monroe Eskew for helpful discussion here.} and any two $Col(\omega,\lambda)$-generic extensions satisfy the same sentences.

Now if $\phi$ holds after forcing with $Col(\omega, \lambda)$ for some $\lambda$, then as $L_\kappa$ is elementary in $L$, $\phi$ also holds after forcing over $L_\kappa$ with $Col(\omega, \lambda')$ for some $\lambda' < \kappa$. But if $G(\lambda')$ is generic over $L_\kappa$ for $Col(\omega, \lambda')$, then $L_\kappa[G(\lambda')]$ is elementary in $L[G(\lambda')]$ as $Col(\omega, \lambda')$ is a forcing of size less than $\kappa$, using the fact that $\kappa$ is a cardinal of $L$. So $\phi$ holds after forcing with $Col(\omega, \lambda')$ for some $\lambda' < \kappa$, as desired.

However, $(L[G], Def^L[G])$ does not satisfy the Class-Generic Inner Model Hypothesis because every real in $L[G]$ is set-generic over $L$, and by Jensen coding any model of the Class-Generic Inner Model Hypothesis must have reals which are not set-generic over $L$.
 \end{proof}

Thus, we see how inner model hypotheses are dependent for their content and implications upon what we allow as extensions. Indeed, the difference between set forcing and class forcing is brought out by the following:

\begin{definition}
($\mathbf{ZFC}$) We say that a model $\mathfrak{M}=(M, \in)$ $\kappa$-globally covers $V$ if for every function $f$ (in $V$) with $dom(f) \in M$ and $rng(f) \subseteq M$, there is a function $g \in M$ such that for all $i \in dom(f)$, $f(i) \in g(i)$ and $\mathfrak{M} \models |g(i)| < \kappa$.
\end{definition}

\begin{theorem}
($\mathbf{ZFC}$) \cite{Bukovsky1973a}  Let $\mathcal{V}$ be a transitive model of $\mathbf{ZFC}$, and $\mathfrak{M}$ an inner model of $\mathcal{V}$ definable in $\mathcal{V}$, and $\kappa$ a regular uncountable cardinal in $\mathfrak{M}$. Then $\mathfrak{M}$ $\kappa$-globally covers $\mathcal{V}$ if and only if $\mathcal{V}$ is a $\kappa$-c.c. set-generic extension of $\mathfrak{M}$.\footnote{For further discussion of this theorem, see  \cite{FriedmanFuchinoSakaiFa}.}
\end{theorem}

The theorem highlights that being a set-forcing extension is a relatively restricted kind of extending construction in comparison to others: If one model is a set forcing extension of another (by some $\kappa$-c.c. forcing), then every function in the extension is already $\kappa$-covered by some function in the ground model. There is no such requirement for class forcing. For example, consider an $x \subseteq \kappa$ that is $\kappa$-Cohen generic over $\mathcal{V}$, and let $f$ be the increasing enumeration of $x$. Then, $f$ is not $\kappa$-covered by a function in $\mathcal{V}$. Thus, the class forcing that adds a single $\kappa$-Cohen at every regular $\kappa$ violates global $\kappa$-covering at every $\kappa$. The question of whether or not there could be an analogue of Bukovsk\'{y}'s Theorem for class forcing is currently unresolved, however any such result would have to transcend the notion of $\kappa$-global covering.

Suppose then that we do wish to assert that the $\mathsf{IMH}$ is true of $V$ for \textit{arbitrary} extensions. Then we need to give meaning (in whatever appropriate codification) to the claim that $V$ has various kinds of extension. The intra-$V$ (i.e. internal to $V$) consequences provable from the $\mathsf{IMH}$ may then vary depending on the kinds of extension we can interpret.

How does this play out with respect to the countable transitive model strategy? Here, there is a {\it prima facie} limitation. As it stands, a $\mathfrak{V}$ elementarily equivalent to $V$ is only accurate for {\it first-order} statements about $V$. Because of the inherent incompleteness in second-order properties over $V$, there is no guarantee that $\mathfrak{V}$ perfectly mirrors $V$'s higher-order properties.  This difficulty is especially interesting given that a large part of set theory comprises examining the structural relationships between models. Simply because an axiom is not first-order is not a reason (without significant further argument) to establish that it is not of independent interest.

The situation can be brought into sharper focus by initially considering the behaviour of arbitrary higher-order properties. Letting $\mathcal{C}^\mathfrak{M}$ denote a collection of classes for a model $\mathfrak{M}$ (i.e. to serve as interpretation of the second-order variables), and $(M, \in, \mathcal{C}^\mathfrak{M})$ be the resulting model of second-order set theory, then it is clear that we can have a $\mathfrak{V}$ elementarily equivalent to $V$ for parameter-free first-order truth, but diverging in the higher-order realm. To see this, consider the following set-sized case. Let some universe $\mathcal{V}$ be of the form $(V_\kappa, \in, V_{\kappa + 1})$, where $\kappa$ is inaccessible. Now use $\mathsf{AC}$ to obtain a countable elementary (in the language of $\mathbf{MK}$) submodel, and collapse to obtain some countable and transitive $\mathfrak{U}=(U, \in, \mathcal{C}^\mathfrak{U})$ (note here that both $U$ and $\mathcal{C}^\mathfrak{U}$ are countable). Letting $Def(M)$ denote the set of classes of a model $\mathfrak{M}$ definable with first-order parameters, we then have that $(U, \in, \mathcal{C}^\mathfrak{U})$ and $(V_\kappa, \in, V_{\kappa + 1})$ satisfy $\mathbf{MK}$ class theory, whereas $(U, \in, Def(U))$ and $(V_\kappa, \in, Def(V_\kappa))$ each satisfy the weaker $\mathbf{NBG}$ class theory, despite all four models having the same first-order theory. This shows that though we might have a $\mathfrak{V}$ elementary equivalent to $V$ for first-order truth, it may radically diverge from $V$ in its higher-order properties, possibly even (dependent on the class theory we accept for $V$, and the classes we choose for $\mathfrak{V}$) having `more' classes (in the sense of sentences of class theory satisfied) than $V$.

 While we see that $\mathfrak{V}$ need not resemble $V$ with respect to arbitrary higher-order truth, there is the question of whether or not it need resemble $V$ with respect to the higher-order truth relevant for satisfaction of the $\mathsf{IMH}$. We will show how to code outer models of $V$ later, in order to give meaning to this question concerning the $\mathsf{IMH}$. However, even in the case of countable transitive models, the question does not have an obvious answer:

 \begin{question}\label{IMH1}
  Could there be a countable transitive model $\mathfrak{M}=(M, \in, C^\mathfrak{M})$ of $\mathbf{NBG}$, such that $\mathfrak{M}$ has a countable (from the perspective of $\mathfrak{M}$) transitive submodel 
  \sloppy $\mathfrak{M}'=(M', \in, C^{\mathfrak{M}'})$, 
  also a model of $\mathbf{NBG}$, with $\mathfrak{M}$ and $\mathfrak{M}'$ agreeing on parameter-free first-order truth in $\mathbf{ZFC}$ but disagreeing on the $\mathsf{IMH}$?
 \end{question}

The mathematics surrounding this question seems difficult, yielding a potentially philosophically and mathematically interesting line of inquiry. For now, we merely point out that it is at least an open (epistemic) possibility that  a $\mathfrak{V}$ equivalent to $V$ for first-order truth might fail to resemble $V$ with respect to the $\mathsf{IMH}$ when we admit some conception of classes for each (assuming that the $\mathsf{IMH}$ can be given a reasonable interpretation over $V$). Later, we will see that if we allow impredicative class comprehension, this difficulty can be circumvented. A small amount of impredicative similarity between $V$ and $\mathfrak{V}$ is sufficient to yield enough resemblance for the satisfaction of the $\mathsf{IMH}$ to covary between $V$ and $\mathfrak{V}$.

\subsection{$\sharp$-generation}

Interestingly, width extensions (i.e. universes containing the same ordinals but more subsets) allow us to encapsulate many large cardinal consequences of reflection properties. It is here that the notion of a {\it sharp} becomes useful. Before we give the definition, we will require a notion of iteration in class theory. We therefore need a preliminary:

\begin{definition}
($\mathbf{NBG}$) Let $\mathsf{ETR}$ (for `{\bf E}lementary {\bf T}ransfinite {\bf R}ecursion') be the statement that every first-order recursive definition along any well-founded binary class relation has a solution.\footnote{For discussions of $\mathsf{ETR}$, see \cite{Fujimoto2012a} and \cite{GitmanHamkins2017a}.}
\end{definition}

We make the following definition (and provide a visual representation in Figure 2):

\begin{definition}
($\mathbf{NBG} + \mathsf{ETR}$) A transitive structure $\mathfrak{N} = (N,U)$ is called a {\it class-iterable sharp with critical point $\kappa$} or just a {\it class-iterable sharp} iff\footnote{This way of defining sharps is modified from the discussion in \cite{Friedman2016a} and \cite{FriedmanHonzik2016a}.}:

\begin{enumerate}[(i)]
\item $\mathfrak{N}$ is a transitive model of $\mathbf{ZFC}^-$ (i.e. $\mathbf{ZFC}$ with the power set axiom removed) in which $\kappa$ is the largest cardinal and is strongly inaccessible.
\item $(N, U)$ is amenable (i.e. $x \cap U \in N$ for any $x \in N$).
\item $U$ is a normal measure on $\kappa$ in $(N, U)$.
\item $\mathfrak{N}$ is iterable in the sense that all successive ultrapower iterations along class well-orders (over the ambient model containing the sharp) starting with $(N, U)$ are well-founded, providing a sequence of structures $(N_i, U_i)$ (for $i$ a set or class well-order) and corresponding $\Sigma_1$-elementary iteration maps $\pi_{i,j}: N_i \longrightarrow N_j$ where $(N, U) = (N_0, U_0)$.
\end{enumerate}
\end{definition}

\begin{remark}
The original work of \cite{Friedman2016a} and \cite{FriedmanHonzik2016a} defined sharps in terms of `all' successive ultrapowers being well-founded in a height potentialist framework (i.e. where any universe can be extended in height). Since we are shooting for axioms on the Universist framework, we will allow the iteration of the ultrapower along any class well-order, and hence make the definitions in $\mathbf{NBG} + \mathsf{ETR}$. In any case, we will see later (\S6) that $\mathsf{ETR}$ can be eliminated in the presence of some impredicative class theory.
\end{remark}

\begin{figure}\caption{A visual representation of the initial ultrapowers corresponding to a sharp $(N,U)$}
 \begin{center}
 \begin{tikzpicture}
  \draw (0,0) -- (-1.5, 3);
  \draw (0,0) -- (1.5,3);
  \node (a) at (0,3) {$\mathfrak{N}=(N, U)$};
  \node (b) at (-1.5,2.3) {$\cdot U$};
  \node (c) at (1.2, 2) {$\kappa$};
  \draw (-1,2) -- (1,2);
  \node (d) at (0,2) {$\bullet$};
  \draw (4,0) -- (2,4);
  \draw (4,0) -- (6,4);
  \node (e) at (4,4.5) {$(N_1 , U_1)$};
  \draw (5.75,3.5) -- (2.25,3.5);
  \node (e) at (4, 3.5) {$\bullet$};
  \node (f) at (1.8,3.8) {$\cdot U_1$};
  \node (g) at (6,3.5) {$\kappa_1$};
  \draw[->] (d) to node[above] {$\pi_{0,1}$} (e);
  \draw[->] (0,1) to node[above] {$\pi_{0,1}$} (4,1);
  \draw[->] (e) to node[above] {$\pi_{1,2}$} (8,5.5);
  \draw[->] (4.2,1) to node[above] {$\pi_{1,2}$} (8,1);
  \node (g) at (8,0) {$\bullet$};
 \node (h) at (8.5,0) {$\bullet$};
 \node (i) at (9,0) {$\bullet$};
 \end{tikzpicture}
 \end{center}
\end{figure}
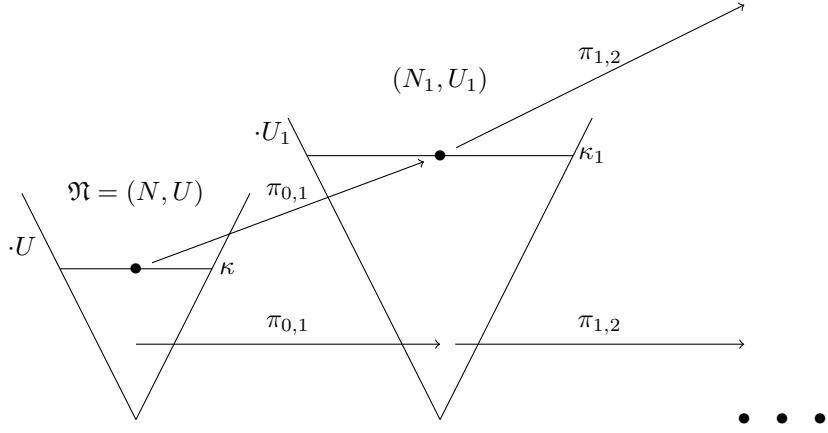

Using the existence of the maps $\pi_{i,j}: N_i \longrightarrow N_j$, we can then provide the following definition:

\begin{definition} ($\mathbf{NBG}+\mathsf{ETR}$) A transitive model $\mathfrak{M} = (M, \in)$ is {\it class iterably sharp generated} iff there is a class-iterable sharp $(N, U)$ and an iteration $N_0 \longrightarrow N_1 \longrightarrow N_2 ...$ such that $M = \bigcup_{\beta \in On^\mathfrak{M}} V^{N_\beta}_{\kappa_\beta}$.
\end{definition}

In other words, a model is class iterably sharp generated iff it arises through collecting together the $V^{N_i}_{\kappa_i}$ (i.e. each level indexed by the largest cardinal of the model with index $i$) resulting from the iteration of a class-iterable sharp through the ordinal height of $\mathfrak{M}$. Note than in producing the model, we only require that the sharp is iterated $Ord^\mathfrak{M}$-many times, despite the fact that it \textit{can} be iterated far further.\footnote{Elsewhere the third author argued that \textit{any} maximal reflection principle will require a formulation external to the model in which it is witnessed, since one can always strengthen any internal reflection principle $\Phi$ with the statement $\Phi$+``There exists an $\alpha$ such that $V_\alpha \models \Phi$''. See \cite{Friedman2016a}, \S4.4.}

A model's being class iterably  sharp generated engenders some pleasant features. In particular, it implies that any satisfaction (possibly with parameters drawn from $\mathfrak{M}$) obtainable in height extensions of $\mathfrak{M}$ adding ordinals (through the well-orders in the class theory of the ambient universe) is already reflected to an initial segment of $\mathfrak{M}$.\footnote{See \cite{Friedman2016a} and \cite{FriedmanHonzik2016a} for discussion.} In this way, we are able to coalesce many reflection principles into a single property of a model. For example a model $\mathfrak{M}$ being class iterably sharp generated already entails reflection from $\mathfrak{M}$ to initial segments of $n$\textsuperscript{th}-order logic for any $n$.\footnote{See here \cite{FriedmanHonzik2016a} and \cite{Friedman2016a}.} One might then suggest the following kind of axiom in attempting to capture reflection properties:

\begin{axiom}
(Informal)\footnote{We will show how to code this axiom formally in \S\S4--6.} {\it The Class Iterable Sharp Axiom.} $V$ is class iterably sharp generated.
\end{axiom}

{\noindent}which would allow us to assert in one fell swoop that $V$ satisfies many reflection axioms (rather than having to assert them in a piecemeal fashion). However, such an axiom is also clearly problematic from the Universist perspective; claiming that $V$ is class iterably sharp generated depends upon the existence of an iterable class sharp for $V$, which cannot be in $V$. If it were, one could obtain a class club resulting from iterating the sharp (namely the class of $\kappa_i$), which in turn forms a club of \textit{regular} $V$-cardinals. The $\omega$\textsuperscript{th} element of any club of ordinals with proper initial segments in $V$ must be {\it singular} with cofinality $\omega$, and so we would obtain a contradiction at $\kappa_\omega$; it would have to be both regular and singular.

Thus, the claim that $V$ is class iterably sharp generated comes out as trivially false; there simply is no such sharp. Moreover, sharps are especially problematic as they cannot be reached by known forcing constructions, so many of our usual simulations of satisfaction for extensions (such as the use of a forcing relation or Boolean-valued model) will not help in their discussion.\footnote{See \cite{Friedman2000a}, \S5.2 for details.} Moreover, whilst the first-order consequences of $\mathfrak{V}$ being class iterably sharp generated will be mirrored in $V$, there is not yet any interpretation of the embeddings yielding the sharp (these being higher-order objects). As we'll see in \S\S 4,5, and 6, we can develop an interpretation of these objects. For now, we look at some other possible uses of extensions.

\subsection{Set-theoretic geology}

A final kind of higher-order axiom that uses extensions emerges from {\it set-theoretic geology}.\footnote{See \cite{FuchsHamkinsReitz2015a} for a description of set-theoretic geology.}  The study of the {\it geological} properties of a model concerns how different models could have arisen by some variety of extension construction (the metaphor of {\it geology} indicates tunnelling {\it down} into the ground models of a particular starting model). In the original paper of \cite{FuchsHamkinsReitz2015a}, the authors are almost exclusively concerned with set-generic forcing extensions. We might examine, however, how this can be extended to {\it arbitrary} extensions. We now repeat some definitions of \cite{FuchsHamkinsReitz2015a} to show how this can be done:

 \begin{definition}
  ($\mathbf{ZFC}$) \cite{FuchsHamkinsReitz2015a} A class $W$ is a {\it ground} of $V$ iff $V$ is obtained by set forcing over $W$, that is if there is some $\mathbb{P} \in W$ and $W$-$\mathbb{P}$-generic filter such that $V = W[G]$.
 \end{definition}

 \begin{definition}
($\mathbf{ZFC}$)  \cite{FuchsHamkinsReitz2015a} A class $W$ is a {\it bedrock} for $V$ iff it is a ground of $V$ and minimal with respect to the forcing-extension relation.
 \end{definition}

\begin{definition}
($\mathbf{ZFC}$) \cite{FuchsHamkinsReitz2015a} The mantle $M$ of a model of set theory is the intersection of all its grounds.
 \end{definition}

\cite{FuchsHamkinsReitz2015a} then goes on to prove several facts about the geological properties models may possess. In particular, the paper shows that many of these statements, which appear second-order (given their reference to structural interrelations of proper class models), can actually be rendered in first-order terms. For example they discuss the following:

 \begin{theorem}
 ($\mathbf{ZFC}$) \cite{Laver2007a} (The Ground Model Definability Theorem). Every ground model $W$ of $V$ is definable in $V$ using a parameter from $W$. Moreover, there is a specific formula $\phi(y,x)$ such that if $W$ is a ground of $V$, then there exists an $r \in W$ such that:

  $$W= \{ x | \phi(r, x) \}$$

 \end{theorem}

This theorem facilitates the study of the geological structure of set-generic multiverses from within a given model in first-order terms. However, we might consider generalisations of the idea of geology to other extensions. \cite{FuchsHamkinsReitz2015a} goes some way towards this, considering the structure present when we allow {\it pseudo-grounds} into the picture: models that have certain covering and approximation properties that facilitate the definability of the ground model in the (possibly class) forcing extension.

Despite this, the ground-model definability theorem can fail badly when we fully admit class forcing extensions into the picture, as shown by the following:

 \begin{theorem}
($\mathbf{MK}$) \cite{Antos2018a} There is a $\mathbf{ZFC}$-preserving class forcing such that the ground model $\mathfrak{M}$ is not definable in the extension $\mathfrak{M}[G]$.
 \end{theorem}

The theorem shows that for many class forcings, the definability of the ground model in the extension can fail. Thus, by insisting that we examine any geological structure only in cases where we {\it do} have definability, we lose a possible route for discovering facts about $V$. Moreover, insisting on the use of set forcing obscures other possible routes of inquiry. For example, instead of looking at forcing extensions, one could examine arbitrary extensions which preserve cardinals. This is a deep and challenging form of set-theoretic geology. Or, one can look at inner models obtained by iterating the $HOD$-operation (i.e. looking at the $HOD$ of the $HOD$, the $HOD$ of the $HOD$ of the $HOD$ etc.).

How might the use of geology be useful in formulating axioms for $V$? A natural thought for a Universist is that $V$ cannot be `reached' by particular kinds of construction. In the case of the set-generic multiverse, we might then postulate that the following is true:

 \begin{definition}
($\mathbf{ZFC}$)  \cite{FuchsHamkinsReitz2015a} $V$ satisfies the {\it ground axiom} iff it is not a non-trivial set-forcing extension of an inner model.
 \end{definition}

In the present case, we might consider the following sort of axiom:

\begin{definition}
(Informal) $V$ satisfies the {\it $\Phi$-ground axiom} iff for extensions of kind $\Phi$ (e.g. set forcing, class forcing, appropriate arbitrary extensions), $V$ is not a non-trivial extension of kind $\Phi$ of an inner model.
 \end{definition}

Axioms of this form would provide a formalisation of the idea that $V$ is `unreachable' in some sense: it cannot be obtained from an inner model by certain kinds of extension. 
As it stands, we do not need to talk about extensions in order to talk about the geological structure of $V$; most of the discussion here concerns how $V$ can be obtained by extending {\it inner} models. However, often geological structure is elucidated by considering how particular grounds (possibly satisfying some $\Phi$-ground axiom) behave within a wider multiverse context. A salient concept here is the following:

 \begin{definition}
 ($\mathbf{ZFC}$) \cite{FuchsHamkinsReitz2015a} The {\it set-generic mantle} is the intersection of all the grounds of set forcing extensions of $V$.\footnote{By the results of \cite{FuchsHamkinsReitz2015a}, this is parametrically definable in the extension.}
 \end{definition}

In the case of set forcing, this class will be definable, owing to the Laver definability of the ground model of an extension. In fact, by recent results of Usuba\footnote{Namely his proof of the {\it Downward Directed Grounds Hypothesis} (that any two grounds have a common ground) and the {\it Set-Downward Directed Grounds Hypothesis} (that a set-sized parameterised family of grounds have a common ground). See \cite{Usuba2017a} for details.}, the mantle (the intersection of all grounds of $V$) and set-generic mantle coincide. However, the loosening of the requirement on set forcing in the definition of the mantle (say to cardinal preserving arbitrary extensions), and the use of this (possibly non-definable) higher-order entity in postulating an axiom to hold of $V$ (for example that the `arbitrary' mantle is a proper subclass of $V$), requires some interpretation of arbitrary outer models of $V$.

\section{The syntactic coding of extensions using $V$-logic}

We thus find ourselves in a tricky situation. We wish to see if we can use extension talk to make claims about $V$, but are at a loss how to do this for certain axioms relating to higher-order properties of $V$. In the rest of the paper, we provide a new method for a Universist to interpret extensions of $V$, and argue that it is philosophically virtuous. First, we provide a sketch of the proposal. 

The central idea will be to use an infinitary proof system ($V$-logic) to code satisfaction in extensions of $V$ syntactically. This has already been mentioned in \cite{AntosFriedmanHonzikTernullo2015a} and \cite{Friedman2016a}, but we provide three important additional contributions. First (\S4) we provide a full and detailed technical account of the mechanisms of $V$-logic, which has not yet appeared. Second (\S5) we will show how $V$-logic relates to impredicative class theories, in particular showing how certain modifications of $V$-logic facilitate a coding in class theory with some impredicative comprehension. In particular, we show how it is possible, given an impredicative class theory, to code consistency in $V$-logic by a single class. While we'll argue that this gives some meaning to the notion of outer model for $V$ (and hence we can provide some interpretation of the axioms we have mentioned) we will also point out that more is needed to satisfy the Methodological Constraint. For our third contribution, we then (\S6) discuss methods for remedying this defect by finding appropriate countable transitive models. For example, we can introduce a predicate for this class and use the resulting theory to reduce our claims to the countable, providing an interpretation that links truth in $V$ to our naive thinking concerning extensions, much as the earlier discussed countable transitive model strategy did for first-order truth. Separately, we will also consider second-order conditions that $V$ might satisfy which yield a {\it first-order} definition of satisfaction in arbitrary outer models.

\subsection{Exposition of $V$-logic}

That infinitary logics relate to talk concerning extensions was known since \cite{Barwise1975a}, and utilised by \cite{AntosFriedmanHonzikTernullo2015a} in providing an interpretation of extension talk in a framework where any universe could be extended to another with more ordinals. We provide a more technically detailed exposition of the system of $V$-logic than has been hitherto provided, and show how it can be captured using impredicative class theories, facilitating a possible line of response on behalf of the Universist. Since we will be showing how to code the logic later, we will leave the theory in which the definitions are formulated for \S\S5--6.

We first set up the language:

 \begin{definition}
  $\mathscr{L}^V_\in$ is the language of $\mathbf{ZFC}$ with the following symbols added:

\begin{enumerate}[(i)]
\item A predicate $\bar{V}$ to denote $V$.
\item A constant $\bar{x}$ for every $x \in V$.
\end{enumerate}
\end{definition}

We can then define $V$-logic:

 \begin{definition}
  {\it $V$-logic} is a system in $\mathscr{L}^V_\in$, with provability relation $\vdash_V$ (defined below) that consists of the following axioms:
 \begin{enumerate}[(i)]
 \item $\bar{x} \in \bar{V}$ for every $x \in V$.
 \item Every atomic or negated atomic sentence of $\mathscr{L}_\in \cup \{\bar{x}| x \in V \}$ true in $V$ is an axiom of $V$-logic.\footnote{An anonymous reviewer points out that one might also wish to add a full first-order satisfaction predicate here. However, the $V$-rule (stated below) will ensure that any first-order sentence true in $V$ is proved to hold in $\bar{V}$ (our symbol for $V$) in $V$-logic. We find it clearer to just put in atomic satisfaction and let the $V$-rule do the rest, but one can also incorporate a satisfaction predicate, if so desired.} 

 \item The usual axioms of first-order logic\footnote{See, for example, \cite{Enderton1972a}, p. 112.} in $\mathscr{L}^V_\in$.
 \end{enumerate}
 
For sentences in $\mathscr{L}^V_\in$, $V$-logic contains the following rules of inference:
 
 \begin{enumerate}[(a)]
  \item {\it Modus ponens}: From $\phi$ and $\phi \rightarrow \psi$ infer $\psi$.
  \item {\it The Set-rule}: For $a, b \in V$, from $\phi(\bar{b})$ for all $b \in a$ infer $\forall x \in \bar{a} \phi(x)$.
  \item {\it The $V$-rule}: From $\phi(\bar{b})$ for all $b \in V$, infer $\forall x \in \bar{V} \phi(x)$.
 \end{enumerate}
 \end{definition}

Proof codes in $V$-logic are thus (possibly infinite) well-founded trees with root the conclusion of the proof. Whenever there is an application of the $V$-rule, we get proper-class-many branches extending from a single node. More formally, we define the notion of a {\it proof code} in $V$-logic (an example of which is visually represented in Figure \ref{vlogicdiagram}) as follows:
 
 \begin{definition}
A {\it proof code of $\chi$ in $V$-logic} is a (possibly infinite) well-founded tree, with root the conclusion of the proof (i.e. $\chi$) and where previous nodes are either codes of axioms of $V$-logic or follow from one of its inference rules. Since we will be proving facts about codes of these proofs later, we provide the following inductive definition:

\begin{enumerate}[(A)]
\item Base cases:
\begin{enumerate}[(i)]
\item For every $x \in V$, the tree that has as its domain the code of $\bar{x} \in \bar{V}$ and as its relation $\emptyset$ is a proof in $V$-logic.
 \item For every atomic or negated atomic sentence $\phi$ of $\mathscr{L}_\in \cup \{\bar{x}| x \in V \}$ true in $V$, the tree that has as its domain the code of $\phi$ and as its relation $\emptyset$ is a proof in $V$-logic.
 \item The trees that have as their domains a single axiom of first-order logic in $\mathscr{L}^V_\in$ and as their relations the empty relation are all proofs in $V$-logic.
 \item If we are proving from some set of premises $\mathbf{T}$, each tree with domain a single sentence of $\mathbf{T}$ and empty relation is a proof in $V$-logic.
 \end{enumerate}
\item Inductive steps:
\begin{enumerate}[(i)]\setcounter{enumi}{3}
\item If $\mathbb{T}_\phi$ and $\mathbb{T}_{\phi \rightarrow \chi}$ are proofs in $V$-logic, then the tree obtained by joining the code of $\chi$ as the root to the two trees $\mathbb{T}_\phi$ and $\mathbb{T}_{\phi \rightarrow \chi}$ is a proof in $V$-logic.
\item If $a$ is a set, and we have a non-empty set of proof trees of the form $\mathbb{T}_{\phi(\bar{b})}$ coding proofs of $\phi(\bar{b})$  for all $b \in a$, then the tree that has as a root the code of $\forall x \in \bar{a} \phi (x)$, and all $T_{\phi(\bar{b})}$ extending from that node is also a proof in $V$-logic.
\item If we have a non-empty class of proof trees of the form $\mathbb{T}_{\phi(\bar{b})}$ coding proofs of $\phi(\bar{b})$  for all $b \in V$, then the tree that has as a root the code of $\forall x \in \bar{V} \phi (x)$ and all $\mathbb{T}_{\phi(\bar{b})}$ extending from that node is also a proof in $V$-logic.
 \end{enumerate}
\end{enumerate}
 \end{definition}

\begin{definition}
For a theory $\mathbf{T}$ and sentence $\phi$ in the language of $V$-logic, we say that $\mathbf{T} \vdash_V \phi$ iff there is a proof code of $\phi$ in $V$-logic from $\mathbf{T}$. 
 We furthermore say that a set of sentences $\mathbf{T}$ is {\it consistent in $V$-logic} iff $\mathbf{T} \vdash_V \phi \wedge \neg \phi$ is false for all formulas of $\mathscr{L}^V_\in$.
\end{definition}
 
 \begin{remark}
 We discuss how these proof codes relate to admissible sets (defined in \S5) and can be coded in the class theory later (\S5). The eagle-eyed reader may notice that we do not require that there is only one proof of $\phi(\bar{b})$ for every $b \in V$ at a particular level of the tree, but that proofs can be repeated for some $b$. This differs from previous presentations of $V$-logic\footnote{See, for example, \cite{AntosFriedmanHonzikTernullo2015a}.} (or $\mathfrak{M}$-logic in the terminology of \cite{Barwise1975a}), and will be essential for some interactions between $Hyp(V)$ and the class-theoretic coding (we define and code $Hyp(V)$ below; one can think of it as a class giving meaning to the term ``the least admissible containing $V$ as an element''). There we are not guaranteed the existence of a well-ordering of $Hyp(V)$ and thus cannot a priori pick `least' proof codes. We regard it as a benign feature of our modified definition of $V$-logic proof codes that they can be applied to more structures without problem.
 \end{remark}
 
 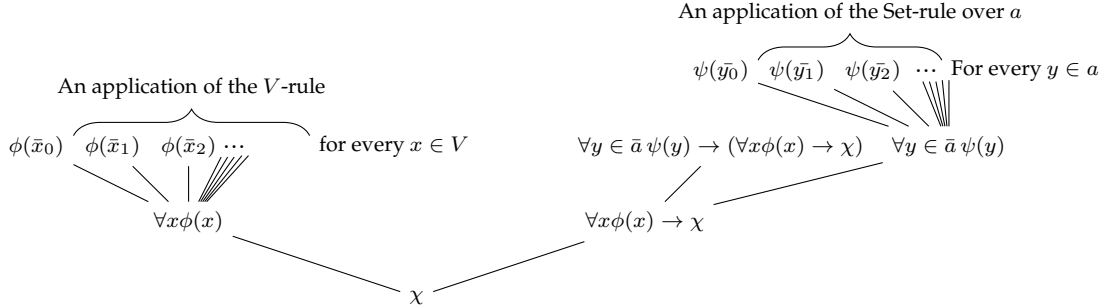
\begin{figure}\caption{Visual representation of a proof of $\chi$ in $V$-logic}\label{vlogicdiagram}
 \vspace{0.1cm}
  \begin{tikzpicture}
  \node (z) at (8,3) {\footnotesize For every $y \in a$};
  \node (y) at (7,3) {};
  \node (x) at (6.9,3) {};
  \node (w) at (6.8,3) {$.$};
  \node (v) at (6.7,3) {$.$};
  \node (u) at (6.6,3) {$.$};
  \node (t) at (6,3) {\footnotesize $\psi (\bar{y_2})$};
  \node (s) at (5,3) {\footnotesize $\psi (\bar{y_1})$};
  \node (r) at (4,3) {\footnotesize $\psi (\bar{y_0})$};
  \node (q) at (7,2) {\footnotesize $\forall y \in \bar{a} \, \psi (y)$};
  \node (p) at (4,2) {\footnotesize $\forall y \in \bar{a} \, \psi (y) \rightarrow (\forall x \phi(x) \rightarrow \chi)$};
  \node (o) at (-0.3,2) {\footnotesize for every $x \in V$};
  \node (k) at (-2.1,2) {};
  \node (j) at (-2.2,2) {};
  \node (i) at (-2.3,2) {$.$};
  \node (h) at (-2.4,2) {$.$};
  \node (g) at (-2.5,2) {$.$};
  \node (f) at (-3,2) {\footnotesize $\phi(\bar{x}_2)$};
  \node (e) at (-4,2) {\footnotesize $\phi(\bar{x}_1)$};
  \node (d) at (-5,2) {\footnotesize $\phi (\bar{x}_0)$};
  \node (c) at (3,1) {\footnotesize $\forall x \phi (x) \rightarrow \chi$};
   \node (b) at (-3,1) {\footnotesize $\forall x \phi (x)$};
    \node (a) at (0,0) {\footnotesize $\chi$};
    \draw (a) -- (c);
  \draw (a) -- (b);
  \draw (b) -- (d);
  \draw (b) -- (e);
  \draw (b) -- (f);
  \draw (b) -- (g);
  \draw (b) -- (h);
  \draw (b) -- (i);
  \draw (b) -- (j);
  \draw (b) -- (k);
  \draw (c) -- (p);
  \draw (c) -- (q);
  \draw (q) -- (r);
  \draw (q) -- (s);
  \draw (q) -- (t);
  \draw (q) -- (u);
  \draw (q) -- (v);
  \draw (q) -- (w);
  \draw (q) -- (x);
  \draw (q) -- (y);
\draw [decorate,decoration={brace,amplitude=10pt,raise=4pt},yshift=0pt]
(d) -- (o) node [black,midway,above, yshift=0.5cm] {\footnotesize An application of the $V$-rule};
\draw [decorate,decoration={brace,amplitude=10pt,raise=4pt},yshift=0pt]
(r) -- (z) node [black,midway,above, yshift=0.5cm] {\footnotesize An application of the Set-rule over $a$};
  \end{tikzpicture}
 \end{figure}

With the mechanisms of $V$-logic set up, we now describe how its use is relevant for interpreting extensions of universes.

\subsection{Interpreting extensions in $V$-logic}

We have provided explanation of a logic that rigidifies the structure of $V$; adding constants and axioms to fix $V$'s properties within the syntax of $V$-logic. How might this allow us to interpret satisfaction in outer models of $V$? As we shall argue, {\it consistency} in $V$-logic of theories in $\mathscr{L}^V_\in$ serves that purpose.\footnote{This idea is also discussed (but only for set-sized models) in \cite{AntosFriedmanHonzikTernullo2015a}, \cite{Friedman2016a}, and \cite{BartonFriedman2017a}. In this paper, we expand these results to the Universist perspective, and show how to interpret $V$-logic in class theory.} For the moment, we will work with $V$-logic as a system in its own right, and show how it can be coded in a manner acceptable to the Universist in \S5.

We first introduce a constant $\bar{W}$ to our language. Letting $\Phi$ be a condition in any particular formal language on universes we wish to simulate in an extension, we then introduce the following `axioms' into our theory of $V$-logic:

\begin{enumerate}[(i)]
 \item {\it $\bar{W}$-Width Axiom.} $\bar{W}$ is a universe satisfying $\mathbf{ZFC}$ with the same ordinals as $\bar{V}$ and containing $\bar{V}$ as a proper subclass.

\item {\it $\bar{W}$-$\Phi$-Width Axiom.} $\bar{W}$ is such that $\Phi$.
\end{enumerate}

We can then have the following axiom to give meaning to the notion of an extension such that $\Phi$, and hence yield intra-$V$ consequences of said extension:

\begin{description}
\item[$\Phi^{\vdash_V}$-Axiom.]The theory in $V$-logic with the $\bar{W}$-Width Axiom and $\bar{W}$-$\Phi$-Width Axiom is consistent under $\vdash_V$.\footnote{Strictly speaking, this will involve a new consequence relation $\vdash_V'$, that includes mention of any axioms involving $\bar{W}$. In fact, \textit{any} collection of additional axioms will result in a new consequence relation involving those axioms. The consequence relation is simply $\vdash_V$ but with any additional axioms added to our original definition of $V$-logic. For clarity we suppress this detail, continue to use  $\vdash_V$ (thereby mildly abusing notation), and show how these relations can be coded formally later.}
\end{description}

We can use this axiom to give meaning to the notion of an intra-$V$ consequence of the axiom mentioning extensions. Any syntactic consequence concerning either some $\bar{x}$ or $\bar{V}$ derived from the axioms mentioning $\bar{W}$ will hold of the respective {\it actual} structures: we simply trace the consequences to the relevant constant. 

To see this in concrete contexts, let us examine $V$-logic in action with respect to some of the examples outlined earlier (we'll show how to code $V$-logic in \S5, but proceed intuitively for now). In the case of a set forcing we could have the following:

\begin{definition}
{\it $\bar{W}$-$G$-Width Axiom.} $\bar{W}$ is such that it contains some $\bar{V}$-$\bar{\mathbb{P}}$-generic $G$.
\end{definition}

If the resulting $V$-logic theory is consistent, then any syntactic consequence of the existence of $\bar{W}$ concerning $\bar{V}$  will then be true in $V$. 

The situation with class forcing is similar, but with a small twist. For, in the case of class forcing using some class poset $\mathbb{P}^C$, the existence of a $V$-$\mathbb{P}^C$ generic $G^C$ is not a {\it first-order} property of $\bar{W}$. Despite this, in $V$-logic we have the ability to add predicates (as we did with $\bar{V}$ and $\bar{W}$). Thus, we can add additional predicates $\bar{\mathbb{P}^C}$ and $\bar{G^C}$ for $\mathbb{P}^C$ and $G^C$ into the usual syntax of $V$-logic, and state the following axiom: 

\begin{definition}
{\it $\bar{W}$-$G^C$-Width Axiom.} $\bar{W}$ is such that $\bar{G^C} \subseteq \bar{W}$ and $\bar{G^C}$ is $\bar{\mathbb{P}^C}$-generic over $V$.
\end{definition}

\noindent and then examine whether the resulting theory is consistent in $V$-logic. Any intra-$V$ consequence of such a (consistent) theory would, for exactly the same reasons as in the case of set forcing, naturally transfer to truths concerning $V$.

This liberation of methods via syntactic means also allows us to formulate axioms that capture {\it non}-forcing extensions. For example:

\begin{definition}
{\it $\bar{W}$-Class-$\sharp$-Width Axiom.} $\bar{W}$ has the same ordinals as $\bar{V}$, satisfies $\mathbf{NBG}+\mathsf{ETR}$, and contains a class sharp that generates $V$.
\end{definition}

\noindent This then allows us to express the claim that $V$ is sharp generated:

\begin{definition}
{\it The Class Iterable Sharp Axiom${}^{\vdash_V}$.} The theory in $V$-logic with the $\bar{W}$-Width Axiom and $\bar{W}$-Class-$\sharp$-Width Axiom is consistent under $\vdash_V$.
\end{definition}

Again, anything provable about $V$ using this theory will be represented by the relevant constants in the theory of $V$-logic, allowing us to give meaning to the claim that $V$ is class iterably sharp generated. Later (\S6) we will see that this corresponds in a neat way to the \textit{actual} existence of sharps over certain (countable) models of set theory.

We noted earlier (\S2) that much extension talk could be interpreted by conducting the construction over a countable transitive model $\mathfrak{V}$ that satisfied exactly the same parameter-free first-order sentences as $V$. It was noted there, however, that the production of such a model provided no guarantee that the model would respect greater than first-order features of $V$, in particular the existence of many inner models provided by the $\mathsf{IMH}$. The key fact here is that now we have the notion of interpreting extensions via consistency in $V$-logic, we are able to simulate statements about the existence of \textit{arbitrary} models and their interrelations.

Again, we add a constant $\bar{W}$ to our language and formulate axioms concerning width extensions represented syntactically by the relevant $\bar{W}$. We can then express the intended content of the $\mathsf{IMH}$ as follows:

\begin{definition}
$\mathsf{IMH}^{\vdash_V}$. Suppose that $\phi$ is a parameter-free first-order sentence. Let $\mathbf{T}$ be a $V$-logic theory containing the $\bar{W}$-Width Axiom and also the $\bar{W}$-$\phi$-Width Axiom (i.e. $\bar{W}$ satisfies $\phi$). Then if $\mathbf{T}$ is consistent under $\vdash_V$, there is an inner model of $V$ satisfying $\phi$.\footnote{The eagle-eyed reader will notice that the formulation in the $\mathsf{IMH}^{\vdash_V}$ is somewhat different from the $\sf IMH$, stating that anything true in an outer model of $V$ is true in an inner model of $V$, rather than anything true in an \textit{inner model of} an outer model is true in $V$. In this way, we have formulated a $V$-logic version of the following principle:

\begin{description}
\item[Definition.] The $\mathsf{IMH}'$ is the claim that if $\phi$ is true in an outer model of $V$, then $\phi$ is true in an inner model of $V$.
\end{description}

One way to rectify this is (as an anonymous reviewer points out) to introduce further predicate letters to stand for the inner model of the outer model. However, it is somewhat simpler (in terms of formulating the $\mathsf{IMH}^{\vdash_V}$) to simply note that:

\begin{description}
\item[Claim.] The $\mathsf{IMH}'$ and the $\mathsf{IMH}$ are equivalent.
\end{description}

\begin{proof}
Clearly the $\sf IMH$ implies the $\mathsf{IMH}'$, since every outer model is also an inner model of an outer model (i.e. the $\mathsf{IMH}'$ is a sub-principle of the $\mathsf{IMH}$). To go in the other direction, we need to show that under the $\mathsf{IMH}'$, if $\phi$ holds in an inner model $W_0$ of an outer model $W_1$ of $V$ then $\phi$ holds in an inner model of $V$. We use Jensen coding to enlarge $W_1$ further to an outer model $W$ in which $W_0$ is a definable inner model (without parameters). We then know that $W_0$ is defined by some formula $\psi$ in $W$. We can then apply the $\mathsf{IMH}'$ to the sentence ``$\psi$ defines an inner model of $\phi$''. This holds in $W$, and therefore by $\mathsf{IMH}'$ there is an inner model of $V$ with $\phi$.
\end{proof}
}
\end{definition}

Thus, by interpreting the existence of outer models through the consistency of theories, we can now make claims concerning consequences (about $V$) of the existence of outer models. In particular, we can say that if $\phi$ is satisfiable in an extension of $V$ (syntactically formulated as $\bar{W}$) then it is satisfied in an inner model of $V$. So, the $\mathsf{IMH}^{\vdash_V}$ holds iff whenever the mathematical structure of $V$ does not preclude the $V$-logic consistency of an outer model satisfying $\phi$, there is an inner model of $V$ satisfying $\phi$. In this way, we make claims concerning greater than first-order properties of $V$ needed to express the $\mathsf{IMH}$ in its maximal sense.

Moreover, we can use this to talk about $V$'s place within a set-theoretically geological structure. Satisfaction in the relevant kinds of extension can be formulated as the consistency of some particular $V$-logic theory or theories, and $V$'s geological properties can thereby be studied. Thus $V$-logic also allows us to talk about $V$ within a higher-order multiverse structure. In fact:

\begin{remark}
We can even interpret satisfaction in non-well-founded extensions of $V$, simply by adding an axiom that states that the extension is non-well-founded and includes $V$ as a standard part.
\end{remark}

Thus, if we allow the use of $V$-logic, we are able to syntactically code satisfaction in arbitrary extensions of $V$ in which $V$ appears standard, and hence the effects of extensions of $V$ on $V$. 

There are two slight wrinkles here, however. First, while we have defined the system, we have yet to show that it can be formulated in a theory acceptable to the Universist. Second, while we have `given meaning' to the idea of an outer model via the use of $V$-logic, we have yet to show that the idea actually does the job we want it to; namely mimicking the existence of extensions. These issues can be brought into sharper focus by contrasting our current situation with the use of forcing relations in interpreting set forcing, possibly in conjunction with a countable transitive elementary submodel $\mathfrak{V}$ of $V$. There, we know (1) We have the resources in $V$ to talk about the relevant forcing relations (via the use of the forcing language), and (2) Whenever $\mathfrak{V}$ has an extension by some $\mathbb{P}$-$\mathfrak{V}$-generic $G$ to $\mathfrak{V}[G] \models \phi$, there will be a corresponding partial order $\mathbb{P}'$ in $V$, forcing relation $\Vdash_{\mathbb{P}'}$, and $p \in \mathbb{P}'$ such that $p \Vdash_{\mathbb{P}'} \phi$. In this way, the forcing relation is \textit{certified} as an acceptable interpretation of forcing over $V$; as it is mirrored in the countable (in line with the Methodological Constraint). In the next two sections we will show that, given an impredicative class theory, we can be in a similar position but with \textit{arbitrary} extensions (and thus adequately formalise the axioms discussed in \S3).

\section{$V$-logic, admissibility, and class theory}

We are now in a position where we have provided a logical system ($V$-logic) that allows us to interpret the axioms we discussed earlier. However, it remains to show that the Universist can legitimately utilise this philosophical system. We proceed as follows: First, we show how we can code a particular \textit{height} extension of $V$ (namely $Hyp(V)$) using impredicative class theory. Of course such a height extension does not really exist for the Universist, rather we are simply talking about what structures can be \textit{coded} using classes. Next, we'll show how our version of $V$-logic can be represented within this structure.

\subsection{Admissibility}

First, we need to explain the structure we will be coding. We first recall the system of {\it Kripke-Platek} set theory $\mathbf{KP}$:

\begin{definition}
 {\it Kripke-Platek Set Theory} (or simply `$\mathbf{KP}$') comprises the following axioms:
 
\begin{enumerate}[(i)]
 \item Extensionality
 \item Union
 \item Pairing
 \item Foundation
 \item $\Delta_0$-Separation: If $\phi$ is a $\Delta_0$ formula in which $b$ does not occur free, then:
\end{enumerate}

$(\forall a) (\exists b) (\forall x) [x \in b \leftrightarrow (x \in a \wedge \phi(x))]$

 \begin{enumerate}[(i)]
 \setcounter{enumi}{5}
 \item $\Delta_0$-Collection: If $\phi$ is a $\Delta_0$ formula in which $b$ does not occur free:
 \end{enumerate}
 
 $(\forall a) [(\forall x \in a) (\exists y) \phi (x, y) \rightarrow (\exists b) (\forall x \in a) (\exists y \in b) \phi (x,y)]$ 
 \end{definition}

We make a further pair of definitions:

 \begin{definition}
 A set $\mathfrak{N}$ is {\it admissible over $\mathfrak{M}$} iff $\mathfrak{N}$ is a transitive model of $\mathbf{KP}$ containing $\mathfrak{M}$ as an element.
 \end{definition}

 \begin{definition}
 $Hyp(\mathfrak{M})$ is the smallest (transitive) $x$ such that $x$ is admissible over $\mathfrak{M}$.
 \end{definition}

Our interest will be in $Hyp(V)$; the least admissible structure containing $V$ as an element. Of course this structure does not really exist. However, we will show that something isomorphic to $Hyp(V)$ can be \textit{coded} in an impredicative class theory.

\subsection{Class theory}

 We now provide a short explanation of the impredicative class theory that we will be using (a variant of $\mathbf{MK}$). Later, we will show that full $\mathbf{MK}$ (with a certain extra axiom on the existence of isomorphisms) is far more than is required. While a philosophical and mathematical analysis of greater than first-order set theory is merited, considerations of space prevent a full examination. A couple of remarks, however, are in order regarding Universism and impredicative class theory.

There are several options available for the Universist in interpreting proper class discourse over $V$. While one method is to simply regard class talk as shorthand for the satisfaction of a first-order formula (a method which yields no impredicative comprehension), this is not the only possibility. We might, for example, interpret the class quantifiers {\it nominalistically} through the use of  {\it plural} resources (as in \cite{Boolos1984a} and \cite{Uzquiano2003a}). Another approach is to interpret classes through some variety of property theory as in \cite{Linnebo2006b} (or with a little massaging to the set-theoretic context \cite{Hale2013a}). Still further, we might interpret the class quantifiers using {\it mereology} (as in \cite{HorstenWelch2016a}).

Key to each is that there is at least the possibility for motivating {\it some} impredicative class comprehension. For example, the plural interpretation is often taken to motivate full impredicative comprehension for classes.\footnote{See \cite{Uzquiano2003a} for discussion.} Many views of property theory suggest that impredicatively defined properties are acceptable.\footnote{Certainly this is the case for \cite{Hale2013a}, though the focus there is not an interpretation of proper classes. \cite{Linnebo2006b} proposes a theory of properties to facilitate semantic theorising concerning set theory, but also licenses some additional higher-order comprehension.} Again, for mereological views the amount of class comprehension licensed will depend on one's views and axioms concerning mereology, however the possibility is open to allow some impredicativity.

For simplicity, we will simply adopt full Morse-Kelly class theory (henceforth `$\mathbf{MK}$'), and discuss how much is required for our purposes later (for most applications, we will require $\Sigma^1_1$-Comprehension for classes). We thus might view the present work as not only providing an interpretation of extension arguments over $V$, but also informing what is possible on various class-theoretic approaches. We start by defining our class theory:

\begin{definition}
 $\mathscr{L}_{\mathbf{MK}}$ has two sorts of variables, one for sets (denoted by lower-case Roman letters $x,y,z, x_0,...$) and one for classes (denoted by upper-case Roman letters $X, Y, Z, X_0,...$), and a single non-logical relation symbol $\in$ that holds between two variables of the first sort or between a variable of the first sort and second sort.\footnote{Really, this is not the same as the set membership relation, however we use the same symbol for ease of expression.} $\mathbf{MK}$ consists of the following axioms:
 \begin{enumerate}[(A)]
  \item Set Axioms:
  \begin{enumerate}[(i)]
  \item Set Extensionality
  \item Pairing
  \item Infinity
  \item Union
  \item Power Set
  \end{enumerate}
  \item Class Axioms:
  \begin{enumerate}[(i)]\setcounter{enumi}{5}
  \item Class Extensionality: $(\forall X)(\forall Y) [(\forall z) ( z \in X \leftrightarrow z \in Y) \rightarrow X = Y]$ (i.e.  Classes with the same members are identical).
   \item Foundation: Every non-empty class has an $\in$-minimal element.
  \item Scheme of Impredicative Class Comprehension: \\
  $(\forall X_1), ..., (\forall X_n) (\exists Y)( Y= \{ x | \phi(x, X_1, ..., X_n) \})$,\\
  where $\phi$ is a formula of $\mathscr{L}_{\mathbf{MK}}$ in which $Y$ does not occur free, may contain both set and class parameters, and in which unrestricted quantification over classes and sets is allowed.
  \item Class Replacement: If $F$ is a (possibly proper-class-sized) function, and $x$ is a set, then $ran(F \upharpoonright x)$ is a set (i.e. $\{ F(y) | y \in x \}$ is a set).
  \item Global Choice: There is a class function $F$ such that for every non-empty $x$
  $\exists y \in x \, F(x) = y$. Equivalently, there is a class that well-orders $V$.\footnote{Note that since Global Choice and Class Replacement imply their set-sized incarnations, we do not need to include Choice and Replacement in our Set Axioms.}
  \end{enumerate}
 \end{enumerate}
\end{definition}

Before we delve into the details, we provide a sketch of our strategy. Some of the work and basic ideas occur in \cite{AntosFriedman2017a} and \cite{Antos2015a}. There, however, additional axioms were required for certain applications\footnote{Namely {\it hyper}class forcing. Antos and Friedman use a version of {\it Class Bounding}, equivalent (modulo $\mathbf{MK}$) to $AC_\infty$; a particular kind of choice principle for classes. It is an interesting question (though one we lack the space to address here) whether or not a Universist should accept that such a principle holds of $V$, not least because the principle is necessary for a good deal of mathematical work (such as second-order ultrapowers). The coding goes through far more easily with the principle, but since we do not need it for the purpose of representing extensions of $V$, we show the coding works without it.} of the coding. We will show that within a variant of $\mathbf{MK}$, $Hyp(V)$ can be coded by a single class. We then provide some arguments as to how the theory of $Hyp(V)$ might be reduced to the countable.

\subsection{Coding $Hyp(V)$ in impredicative class theory}

We first need to code the notion of an {\it ordered pair} of classes. Initially, this seems problematic; normal pairing functions on sets are type-raising in the sense that they have the things they pair as members. We will assume that this is not available on the current interpretation; most Universist interpretations of proper classes do not permit proper classes being members and so constructions that are necessarily type-raising are impermissible. Despite this, we can use the abundance of sets within $V$ (in particular the closure of $V$ under pairing) to code pairs using a standard trick:

 \begin{definition}
 ($\mathbf{MK}$) Let $X$ and $Y$ be classes. We define the class that {\it represents the ordered pair of $X$ and $Y$}, or `$REP(\langle X, Y \rangle)$' as follows:
  
  $REP(\langle X, Y \rangle) = \{ \langle z, i \rangle | (z \in X \wedge i = 1) \vee (z \in Y \wedge i =2) \}$
 \end{definition}

Effectively, we talk about coding an ordered pair of classes by tagging all the members of $X$ with $1$ and all the members of $Y$ with $2$, and referring to the resulting class. Moreover, short reflection on the above coding shows that it is easy to generalise this to $\alpha$-tuples by letting $\alpha$ tag $\alpha$-many copies of $V$. We are now able to make use of the following definition (with visual representation of some examples provided in Figures \ref{treediagram1} and \ref{treediagram2}):

 \begin{definition}
  ($\mathbf{MK}$) \cite{Antos2015a}, \cite{AntosFriedman2017a} A pair $\langle M_0, R \rangle$ is a {\it coding pair} iff $M_0$ is a class with distinguished element $a$, and $R$ is a class binary relation on $M_0$ such that:
  
  \begin{enumerate}[(i)]
   \item $\forall z \in M_0 \exists ! n$ such that $z$ has $R$-distance $n$ from $a$ (i.e. for any element $z$ of $M_0$, $z$ is a single finite $R$-distance away from $a$), and
   \item let $\langle M_0, R\rangle \upharpoonright x$ denote the $R$-transitive closure below $x$. Then if $x, y, z \in M_0$ with $y \not = z$, $yRx$, and $zRx$, then $\langle M_0, R \rangle \upharpoonright y$ is not isomorphic to $\langle M_0, R \rangle \upharpoonright z$, and
   \item if $y, z \in M_0$ have the same $R$-distance from $a$, and $y \not = z$, then for all $v$, $vRy \rightarrow \neg v R z$, and
   \item $R$ is well-founded.
  \end{enumerate}
 \end{definition}

\begin{figure}\caption{An example of a tree corresponding to a coding pair showing membership.}\label{treediagram1}
\setlength{\unitlength}{1cm} 
\vspace{1cm}
\begin{picture}(1,1)
\put(2.9,1.4){$\mathbb{T}_x$}

\put(3,1){\line(0,-1){1}}
\put(3,1){\line(-1,-1){1}}
\put(3,1){\line(1,-1){1}}
\put(3,1){\line(-1,-2){0.5}}
\put(3,1){\line(1,-2){0.5}}
\put(3,0){\line(-1,-2){0.5}}
\put(3,0){\line(1,-2){0.5}}
\put(3,0){\line(0,-1){1}}
\put(3.02,1){$a_{x}$}
\put(3.02,0){$a_y$}
\put(3.02,-1){$a_z$}
\put(3.35,-0.5){$\mathbb{T}_y$}

\put(5,0.5){\footnotesize codes}
\put(7.3,1){$x$}
\put(7.3,0.5){$\in$}
\put(7.3,0){$y$}
\put(7.3,-0.5){$\in$}
\put(7.3,-1){$z$}
\end{picture}
\vspace{1cm}
\end{figure}
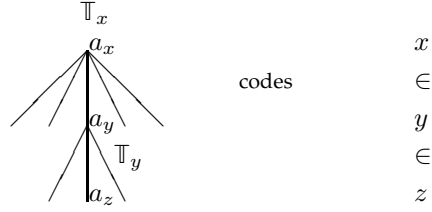

\begin{figure}\caption{More complicated coding pair tree structures.}\label{treediagram2}

 \vspace{1cm}
 \setlength{\unitlength}{1cm} 
\begin{picture}(1,1)
\put(2.9,1.4){$\mathbb{T}_x$}

\put(3,1){\line(0,-1){1}}
\put(3,1){\line(-1,-1){1}}
\put(3,1){\line(1,-1){1}}
\put(3,1){\line(-1,-2){0.5}}
\put(3,1){\line(1,-2){0.5}}
\put(3,0){\line(-1,-1){1}}
\put(3,0){\line(1,-1){1}}
\put(3,0){\line(0,-1){1}}
\put(3,0){\line(-1,-2){0.5}}
\put(3,0){\line(1,-2){0.5}}
\put(2,-1){\line(-1,-2){0.5}}
\put(2,-1){\line(1,-2){0.5}}
\put(4,-1){\line(1,-2){0.5}}
\put(4.5,-2){\line(-1,-2){0.5}}
\put(4.5,-2){\line(1,-2){0.5}}
\put(3.02,1){$a_{x}$}
\put(3.02,0){$a_y$}
\put(1.6,-1){$a_v$}
\put(4.02,-1){$a_w$}
\put(4.52,-2){$a'_v$}
\put(1.8,-2){$\mathbb{T}_v$}
\put(4.35,-3){$\mathbb{T}'_v$}

\put(4.25,0.5){\footnotesize codes}
\put(5.3,1){$x$}
\put(5.3,0.5){$\in$}
\put(5.3,0){$y$}
\put(5.3,-0.5){$\in$}
\put(5,-1){$v$}
\put(5.5,-1){$w$}
\put(5.5,-1.5){$\in$}
\put(5.5,-2){$v$}

\put(8.9,1.4){$\mathbb{T}_x$}

\put(9,1){\line(0,-1){1}}
\put(9,1){\line(-1,-1){1}}
\put(9,1){\line(1,-1){1}}
\put(9,1){\line(-1,-2){0.5}}
\put(9,1){\line(1,-2){0.5}}

\put(9.5,0){\line(-1,-2){0.5}}
\put(9.5,0){\line(1,-2){0.5}}
\put(8,0){\line(-1,-2){0.5}}
\put(8,0){\line(1,-2){0.5}}
\put(7.5,-1){\line(-1,-2){0.5}}
\put(7.5,-1){\line(1,-2){0.5}}
\put(10,-1){\line(-1,-2){0.5}}
\put(10,-1){\line(1,-2){0.5}}
\put(9.02,1){$a_{x}$}
\put(9.52,0){$a_y$}
\put(7.6,-1){$a_v$}
\put(7.6,0){$a_w$}
\put(10.02,-1){$a'_v$}
\put(7.35,-2){$\mathbb{T}_v$}
\put(9.85,-2){$\mathbb{T}'_v$}

\put(10,0.5){\footnotesize codes}
\put(11.3,1){$x$}
\put(11.3,0.5){$\in$}
\put(11.1,0){$w,y$}
\put(11.3,-0.5){$\in$}
\put(11.3,-1){$v$}

\end{picture}

 \vspace{3cm}
 
\end{figure}
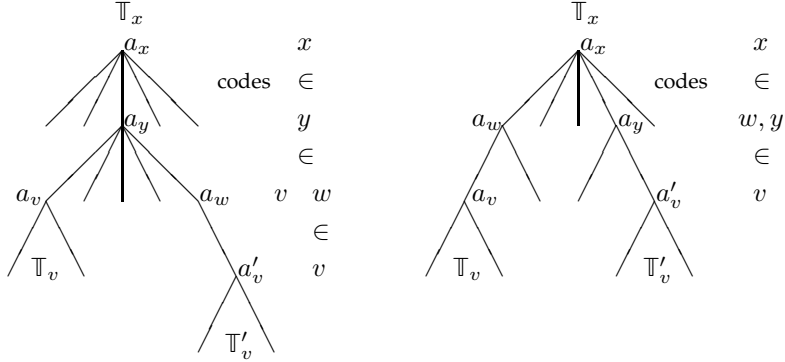

These coding pairs shall be essential in coding the structure of ideal sets that would have to be `above' $V$ were they to exist. One can think of the coding pair as a tree $\mathbb{T}$ which has as its nodes the elements of $M_0$, $a$ as top node, and $R$ the extension relation of $\mathbb{T}$. For each tree there are only countably many levels, but each level can have proper-class-many nodes.

Next, we code the ideal objects using coding pairs. For any particular ideal set $x$, we will code the transitive closure of $\{x\}$. A fact of the above coding is that any tree will have many isomorphic subtrees, and hence will not be isomorphic to $TC(\{x\})$.\footnote{See \cite{Antos2015a} and \cite{AntosFriedman2017a} for details and further explanation. We would like subtrees to correspond to elements in the transitive closure. However, isomorphic subtrees would code the same element, and so as it stands our coding pairs are not extensional. As we'll see below, this can be easily fixed.} We therefore need to form {\it quotient} pairs that provide a coding of ideal sets (we provide a visual representation of an example in Figure \ref{quotientdiagram}).

 \begin{definition}
 ($\mathbf{MK}$)
  \cite{Antos2015a}, \cite{AntosFriedman2017a} {\it Quotient Pairs.} Let $\langle M_0, R \rangle$ be a coding pair and $a$ be a set in $M_0$. We then define the equivalence class of $a$ (denoted by `$[a]$') of all top nodes of the coding trees isomorphic to the subtree $\mathbb{T}_a$:

  \begin{description}
  \item $[a] = \{ b \in M_0 |$``$\langle M_0 , R \rangle \upharpoonright b$ is isomorphic to $\langle M_0, R \rangle \upharpoonright a$''$\}$
  \end{description}

 Since we have Global Choice, we let $\tilde{a}$ be a fixed representative of $[a]$. We then define the {\it quotient pair} $\langle \tilde{M}_0, \tilde{R} \rangle$ as follows:
 
 \begin{description}
 \item $\tilde{M}_0 = \{ \tilde{a} |$ ``$\tilde{a}$ is the representative of the class $[a]$ for all $a \in M_0 \}$
 \end{description}
 \begin{description}
 \item $\tilde{a}\tilde{R}\tilde{b}$ iff ``There is an $a_0 \in [a]$ and a $b_0 \in [b]$ such that $a_0 R b_0$.''
 \end{description}
 \end{definition}

\begin{remark}
 The quotient pairs work by taking fixed representatives of the equivalence class of top nodes of isomorphic subtrees. We then define the relation on these representatives by searching through the equivalence classes to find a relevant subtree in which two members of the equivalence class are $R$-related.
 \end{remark}

\begin{figure}\caption{The quotient process for a coding pair of $3$}\label{quotientdiagram}

\setlength{\unitlength}{1cm} 
\vspace{0.5cm}
\begin{picture}(1,1)

\put(3,0){\line(-1,-1){1}}
\put(3,0){\line(1,-1){1}}
\put(3,0){\line(0,-1){1}}
\put(3,-1){\line(0,-1){1}}
\put(4,-1){\line(-1,-2){0.5}}
\put(4,-1){\line(1,-2){0.5}}
\put(4.5,-2){\line(0,-1){1}}

\put(2.9,0.15){$3$}
\put(2.95,-0.03){$\centerdot$}
\put(1.75,-1){$0$}
\put(1.95,-1.05){$\centerdot$}
\put(2.7,-1){$1$}
\put(2.95,-1.05){$\centerdot$}
\put(2.65,-2){$0'$}
\put(2.95,-2.05){$\centerdot$}
\put(4.1,-1){$2$}
\put(3.95,-1.05){$\centerdot$}
\put(4.6,-2){$1'$}
\put(4.45,-2.05){$\centerdot$}
\put(3.2,-2){$0''$}
\put(3.45,-2.05){$\centerdot$}
\put(4,-3){$0'''$}
\put(4.45,-3){$\centerdot$}

\put(6,-1){$\to$}
\put(9,0){\line(-1,-1){1}}
\put(9,0){\line(1,-3){1}}
\put(9,0){\line(0,-1){2}}
\put(8,-1){\line(1, -1){1}}
\put(9,-2){\line(1,-1){1}}
\put(8,-1){\line(1, -2){0.5}}
\put(8.5,-2){\line(1, -1){0.5}}
\put(9,-2.5){\line(2, -1){1}}

\put(8.9,0.15){$\tilde{3}$}
\put(8.95,-0.04){$\centerdot$}
\put(7.75,-1){$\tilde{2}$}
\put(7.95,-1.05){$\centerdot$}
\put(9.1,-2){$\tilde{1}$}
\put(8.95,-2.05){$\centerdot$}
\put(10.08,-3){$\tilde{0}$}
\put(9.94,-3.02){$\centerdot$}

\end{picture}\\
\vspace{3cm}
\end{figure}
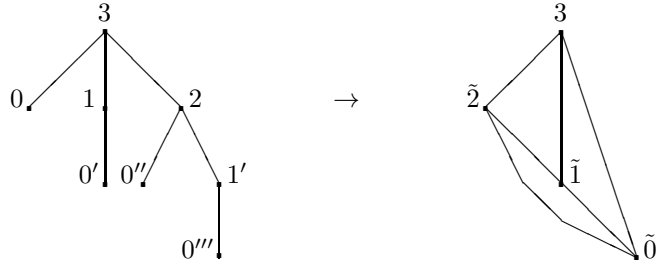

We now have a quotient structure for the coding pairs. Next, we mention some useful properties of these coding pairs for the purposes of showing that they code ideal sets:

 \begin{lemma}\label{codingpairs}
  ($\mathbf{MK}$) \cite{Antos2015a}, \cite{AntosFriedman2017a}\footnote{See \cite{AntosFriedman2017a}, Lemma 2.5.} Let $\langle M_0, R \rangle$ be a coding pair. Then the quotient pair $\langle \tilde{M}_0, \tilde{R} \rangle$ is extensional and well-founded.
 \end{lemma}

Let us take stock. We have quotient structures of coding pairs that behave extensionally and in a well-founded manner, and are coded by individual classes. We are now ready to establish a useful theorem, showing that these quotient coding pairs obey certain operations, and thus we have a code for $Hyp(V)$. We first, however, explain how we will state what we wish to say. 

We will be interested in the theory of these quotient coding pairs, and what codes they can produce. For ease of proof, we will speak as though they are first-order objects, despite the fact that we are using this as an abbreviation for class-theoretic language. We therefore need to make the following:

\begin{remark}
We will talk about `a structure' $(V)^+$, to be understood as the `universe' of all quotient coding pairs, under a `relation' $\widehat{\in}$ that we will show satisfies certain {\it first-order} axioms. We will therefore use locutions like ``$x \widehat{\in} y$'', ``$x \widehat{\in} (V)^+$'', and more generally ``$\phi(x)$'' for some first-order $\phi$. Really, however, since the objects of $(V)^+$ are {\it ideal}, this should all be paraphrased in terms of quotient coding pairs. Though we {\it deny} that we are referring to {\it sets} as normally understood, there is no obstruction to using a first-order language to represent the theory of quotient coding pairs (just as there is no contradiction in using a two-sorted first-order language in our class-theory). We therefore make the following definitions:
\end{remark}

\begin{definition}
($\mathbf{MK}$) Suppose there is a class $X$ coding a quotient pair $\langle \tilde{M}_0, \tilde{R} \rangle$, which in turn codes some $x \widehat{\in} (V)^+$. Since it will often be useful to talk about $X$ coding $\langle \tilde{M}_0, \tilde{R} \rangle$ under their presentation as a tree coding $x$, we say that:

\vspace{0.1cm}

\noindent $\mathbb{T}_x$ is a {\it quotient pair tree for $x$} iff there is a class $X$ representing $\langle \tilde{M}_0, \tilde{R} \rangle$ coding $x$, and $\mathbb{T}_x$ is the tree structure that $X$ exemplifies.
\end{definition}

\begin{definition}
($\mathbf{MK}$) Say there is a class $X$ coding a quotient pair tree $\mathbb{T}_x$. Then $\mathbb{T}_y$ (coded by a class $Y$) is a {\it direct subtree} of $\mathbb{T}_x$ iff $\mathbb{T}_y$ a proper subtree of $\mathbb{T}_x$ and the top node $a_y$ of $\mathbb{T}_y$ is in the level immediately below the top node $a_x$ of $\mathbb{T}_x$. 
\end{definition}

\begin{definition}
($\mathbf{MK}$) Let $X$ and $Y$ code quotient pair trees $\mathbb{T}_x$ and $\mathbb{T}_y$ respectively. Then {\it $X$ bears the $E_T$ relation to $Y$} iff $\mathbb{T}_y$ is isomorphic to a direct subtree of $\mathbb{T}_x$. We shall also write $\mathbb{T}_y E_T \mathbb{T}_x$ to represent this relation between two classes $X$ and $Y$, and will then speak of $y \widehat{\in} x \widehat{\in} (V)^+$.
\end{definition}

\begin{definition}
($\mathbf{MK}$) Let $X$ and $Y$ code quotient pair trees $\mathbb{T}_x$ and $\mathbb{T}_y$ respectively. Then we say $X$ and $Y$ are {\it quotient pair equivalent} (or $\mathbb{T}_x =_T \mathbb{T}_y$) iff $\mathbb{T}_x$ and $\mathbb{T}_y$ are isomorphic.
\end{definition}

\begin{definition}
($\mathbf{MK}$) $(V)^+$ is the structure obtained by taking as our domain of quantification all quotient coding pairs over $V$, our equality relation to be $=_T$, and our membership relation to be $E_T$. We will talk about $(V)^+$ using the following first-order language: A first-order variable $x$ ranges over quotient coding pairs, each variable can be interpreted as a coding pair $\langle \tilde{M}_x, \tilde{R}_x \rangle$, that in turn codes a tree $\mathbb{T}_x$. Membership (denoted by `$\widehat{\in}$') and equality (denoted by `$\widehat{=}$') are interpreted as $E_T$ and $=_T$ respectively.
\end{definition}

We are now in a position to establish a theorem that will prove to be very useful in relating class theory and $V$-logic. Much of the work is putting together results in \cite{Antos2015a} and \cite{AntosFriedman2017a}. However, the following (somewhat lengthy) remark is required to situate the current discussion:

\begin{remark}
In \cite{Antos2015a} and \cite{AntosFriedman2017a}, the first and third author used an axiomatisation $\mathbf{MK}^*$ that also includes the following Class Bounding Axiom:

\begin{definition}

($\mathbf{MK}$) {\it Class Bounding.} $$\forall x \exists A \phi(x,A) \rightarrow \exists B \forall x \exists y \phi (x, (B)_y)$$

\noindent where $(B)_y$ is defined as follows:

$$(B)_y = \{ z | \langle y, z \rangle \in B \}$$

\end{definition}

They also show that, when working over a countable transitive model
\sloppy$\mathfrak{M} = (M, \in, \mathcal{C} )$ 
such that 
\sloppy $\mathfrak{M} \models \mathbf{MK}^*$, 
the model $(M^+, \in)$, where $M^+$ is defined as follows:

$$M^+ = \{ x |\text{``There is a coding pair } \langle M_x , R_x \rangle \text{ for } x \text{ in } \mathcal{C}\text{''}\}$$

\noindent satisfies $\mathbf{SetMK}^*$; a version of $\mathbf{ZFC} -$Power Set with a Set Bounding Axiom and some constraints on the cardinal structure of $M^+$. We indicate how the proofs we require are (a) amenable to the current context, (b) realisable using an impredicative class theory, and (c) can be accomplished avoiding the use of Bounding (but with an assumption on the existence of isomorphisms). For the purposes of the proof, it will be much easier to speak of the tree structures $\mathbb{T}_x$, $\mathbb{T}_y$, and $\mathbb{T}_z$, rather than constantly paraphrasing in terms of classes representing the quotient coding structures.

The proofs in \cite{Antos2015a} and \cite{AntosFriedman2017a} rely on the following two lemmas:

\begin{lemma}
($\mathbf{MK}^*$) \cite{Antos2015a}, \cite{AntosFriedman2017a} {\it First Coding Lemma.} Let $\mathfrak{M} = (M, \mathcal{C})$ be a transitive $\beta$-model (i.e. it computes well-founded relations correctly) of $\mathbf{MK}^*$. Let $\langle N_1 , R_1 \rangle$ and $\langle N_2 , R_2 \rangle$ be coding pairs in $\mathcal{C}$. Then if there is an isomorphism between $\langle N_1 , R_1 \rangle$ and $\langle N_2, R_2 \rangle$ then there is such an isomorphism in $\mathcal{C}$. 
\end{lemma}

\begin{lemma}
($\mathbf{MK}^*$) \cite{Antos2015a}, \cite{AntosFriedman2017a} {\it Second Coding Lemma.} For all $x \in M^{+}$ there is a one-to-one function $f \in M^+$ such that $f: x \longrightarrow M_x$, where $\langle M_x, R_x \rangle$ is a coding pair for $x$.
\end{lemma}

\cite{Antos2015a} and \cite{AntosFriedman2017a} are concerned with higher-order forcing (i.e. using forcing posets that have some {\it proper classes} of a model as conditions) over models of $\mathbf{MK}^*$. In order to deal with the obvious metamathematical difficulties, they {\it explicitly} define the construction over models satisfying $\mathbf{MK}^*$ that are {\it countable}, {\it transitive}, and are $\beta$-models. Their strategy is to code a model $(\mathfrak{M})^+$ of $\mathbf{SetMK}^*$ in the original model of $\mathbf{MK}^*$, and perform a definable class forcing over $(\mathfrak{M})^+$. This then corresponds to a hyperclass forcing over $\mathfrak{M}$. The current work shows that the coding outlined is not dependent upon the countability of the models, nor the extra assumption of Class Bounding (though we will motivate the acceptance of the First Coding Lemma). Instead, we can take the coding over $V$ using the interpretation of $\mathbf{MK}$ through class theory, and show how to code the theory of $(V)^+$ using these resources. For the above lemmas then, a few remarks are in order.

(I) The assumption that the model over which we code is a $\beta$-model is philosophically trivial in the present setting; we are working over the Universist's $V$, with some conception of its classes $\mathcal{C}^V$. For the Universist, $(V, \mathcal{C}^V)$ {\it sets the standard} for what a $\beta$-model is, and so is trivially a $\beta$-model.

(II) Similar remarks apply to the First Coding Lemma, which is a non-trivial result when we are concerned with a countable transitive model $\mathfrak{M} \models \mathbf{MK}^*$. Since $\mathfrak{M}$ has an impoverished view of what classes there are, in that setting one needed to establish that $\mathbf{MK}^*$ satisfaction {\it alone} ensures that there is a class of the relevant kind in $\mathcal{C}$. For the purposes at hand, however, the result is again philosophically trivial; the relevant classes $\mathcal{C}^V$ over $V$ {\it set the standard} for when two trees representing classes are isomorphic, and so we {\it cannot} have isomorphic trees $\mathbb{T}_x$ and $\mathbb{T}_y$ for which there is not a class coding an isomorphism. If there are no things coding an isomorphism between $\mathbb{T}_x$ and $\mathbb{T}_y$ then they are simply {\it not} isomorphic. To be absolutely technically explicit about this fact, we will {\it add} the First Coding Lemma as an extra axiom, and denote the theory we work in by $\mathbf{MK}^+$.

(III) Use of the Second Coding Lemma is circumnavigated by the proofs below, and so we will say nothing more about it.
\end{remark}

 \begin{theorem}\label{maintheorem}
 ($\mathbf{MK}^+$) $(V)^+$ satisfies Infinity, Extensionality, Foundation, Pairing, Union, and $\Delta_n$-Separation for every $n$. Rendered in the class theory, this states:
 
 \begin{enumerate}[(1.)]
  \item {\it Infinity}${}^+$: There is a class representing a tree $\mathbb{T}_\omega$ for $\omega$.
  \item {\it Transitivity}${}^+$: For any class $X$ representing a quotient coding pair tree $\mathbb{T}_x$, and for any class $Y$ representing a quotient coding pair tree $\mathbb{T}_y E_T \mathbb{T}_x$, if $a_z$ is a node directly below $a_y$ in $\mathbb{T}_y$, then there is a class $Z$ coding a direct subtree $\mathbb{T}_z$ of $\mathbb{T}_y$.
  \item {\it Pairing}${}^+$: For any class $X$ coding a tree $\mathbb{T}_x$ and any class $Y$ coding a tree $\mathbb{T}_y$, there is a class $Z$ coding a tree $\mathbb{T}_z$ such that $\mathbb{T}_x E_T \mathbb{T}_z$ and $\mathbb{T}_y E_T \mathbb{T}_z$.
  \item {\it Union}${}^+$: Suppose that there is a class $X$ coding a quotient pair tree $\mathbb{T}_x$. Then there is a class $U_{\cup x}$ coding a quotient pair tree $\mathbb{T}_{\cup x}$, such that for every class $Y$ coding a quotient pair tree $\mathbb{T}_y$ with $\mathbb{T}_y E_T \mathbb{T}_x$, if we have classes coding trees $\mathbb{T}_{z_i} E_T \mathbb{T}_y$, then for each $\mathbb{T}_{z_i}$ we have $\mathbb{T}_{z_i} E_T \mathbb{T}_{\cup x}$.
  \item {\it $\Delta_n$-Separation}${}^+$: Suppose that there is a class $X$ coding a quotient pair tree $\mathbb{T}_x$, in turn coding some $x \widehat{\in} (V)^+$. Let $\phi(y)$ be a  $\Delta_n$ formula in the language of $(V)^+$. Then there is a class $Z$ coding a tree $\mathbb{T}_z$ such that for any class $Y$ coding a tree $\mathbb{T}_y$ (with $y \widehat{\in} (V)^+$) such that $\mathbb{T}_y E_T \mathbb{T}_x$, $\phi(y)$ holds in the theory of $(V)^+$ iff $\mathbb{T}_y E_T \mathbb{T}_z$.
 \end{enumerate}
 \end{theorem}
 
\begin{proof}
The proofs of (1.)--(4.) do not require the use of Class Bounding, and so we refer the reader to \cite{Antos2015a} and \cite{AntosFriedman2017a}. We thus begin with:

(5.) {\it $\Delta_n$-Separation${}^+$}. This is slightly more difficult in that it is not clear how to code a {\it first-order} formula within $(V)^+$. This is dealt with by the following:

\begin{lemma}
 ($\mathbf{MK}^+$) \cite{Antos2015a}, \cite{AntosFriedman2017a} Let $\phi$ be a formula in the language of $(V)^{+}$. Then there is a formula $\psi$ in the language of classes (and the trees they code) such that the theory of $(V)^+$ contains $\phi(x_1, ..., x_k)$ iff $\psi(\mathbb{T}_{x_1}, ..., \mathbb{T}_{x_k})$.
\end{lemma}

\begin{proof}
By induction on the complexity of $\phi$. Suppose $\phi$ is of the form $y \widehat{\in} x$, and let $\mathbb{T}_y$ and $\mathbb{T}_x$ be the associated trees. As $y \widehat{\in} x$, there is a tree $\mathbb{T}_{y'}$ with top node $a_{y'}$ in the level below the top node of $\mathbb{T}_x$, such that $\mathbb{T}_y$ is isomorphic to $\mathbb{T}_{y'}$. By the First Coding Lemma (trivial in the current setting), $\psi$ is then ``$Y$ is a class coding a tree $\mathbb{T}_y$, and $X$ is a class coding a tree $\mathbb{T}_x$ such that $\mathbb{T}_y$ is isomorphic to  a class $Y'$ coding a direct subtree $\mathbb{T}_{y'}$ of $\mathbb{T}_x$''. 
Suppose then that $\phi$ is of the form $y \widehat{=} x$. Then $\psi$ is simply ``The class that codes $\mathbb{T}_y$ is isomorphic to the class that codes $\mathbb{T}_x$'', which is again dealt with by the First Coding Lemma.

For the inductive steps where $\phi$ is of the form $\neg \chi_0$ or\footnote{The choice of $\wedge$ here was somewhat arbitrary, any suitable connective will do.} $\chi_1 \wedge \chi_2$, the result is immediate, we just either negate or conjoin the class-theoretic correlate of the $\chi_m$ provided by the induction step.

Suppose then that $\phi$ is of the form $\forall x \chi$, where $\chi$ is translatable in our class theory by $\chi'$. Then $\psi$ is ``For any class $X$ coding some tree $\mathbb{T}_x$, $\chi'(X)$''. 
\end{proof}

We now can proceed with the proof of $\Delta_n$-Separation${}^+$. Let $a, x_1,..,x_n$ be first-order names in the theory of $(V)^+$ and $\phi(x, x_1,...,x_n, a)$ be any $\Delta_n$-formula in the language of $(V)^+$. We need to show that:

$$b= \{x \widehat{\in} a | \phi(x, x_1,..,x_n,a)\} \widehat{\in} (V)^+$$

\noindent and hence that there is a class $B$ representing a coding pair for $b$ with a corresponding tree $\mathbb{T}_b$.

Let $\mathbb{T}_{x_1}$,...,$\mathbb{T}_{x_n}$ and $\mathbb{T}_a$ be codes, that we will refer to using  $x_1,...,x_n$ and $a$ in the language of $(V)^+$. Further let $\psi$ be the class-theoretic correlate of $\phi$. If $b$ is empty the result is immediate as there is a coding pair for the empty set. Assume then that $b$ is non-empty and $(V)^+$ thinks that $b$ contains some $c_0$ with coding pair tree $\mathbb{T}_{c_0}$. Let $\mathbb{T}_{a(c)}$ be a class variable over trees, with the condition that each $\mathbb{T}_{a(c)}$ corresponds to a direct subtree of $\mathbb{T}_a$ (i.e. $a(c)$ is a member of $a$ in $(V)^+$).  By Class Comprehension, there is a class $Z$ such that if $\psi(\mathbb{T}_{a(c)}, \mathbb{T}_{x_1}, ..., \mathbb{T}_{x_n}, \mathbb{T}_a)$ holds then $\{ z | \langle c, z \rangle \in Z \}$ is the direct subtree $\mathbb{T}_{a(c)}$ of $\mathbb{T}_a$, and if 
\sloppy $\psi(\mathbb{T}_{a(c)}, \mathbb{T}_{x_1}, ..., \mathbb{T}_{x_n}, \mathbb{T}_a)$ 
does not hold then $\{ z | \langle c, z \rangle \in Z \} = \mathbb{T}_{c_0}$.

We then let $\mathbb{T}_b$ be a tree with top node $b_0$ that has as all its direct subtrees the various $\{ z | \langle c, z \rangle \in Z \}$. The tree $\mathbb{T}_b$ then codes the existence of the necessary
\sloppy $b = \{x \widehat{\in} a | \phi(x, x_1,...,x_n, a) \} \widehat{\in} (V)^+$.
\end{proof}

\begin{remark}
Since we are coding in class theory the theory of $(V)^+$, we have `ordinals' that are `longer' than $On$. We will refer to ideal ordinals `past' $V$ (i.e. well-orders in the class theory longer than $On$) using variants of the Greek letters (such as `$\varkappa$', `$\vartheta$', `$\varsigma$', `$\varpi$' etc.). In discussion of second-order set theory, these are also sometimes referred to as `meta-ordinals'.
\end{remark}

With these properties in place, we proceed to find a code of $Hyp(V)$ in $(V)^+$:

\begin{theorem}
($\mathbf{MK}^+$) $(V)^+$ contains a code for $Hyp(V)$.
\end{theorem}

\begin{proof}
We begin with the following definition concerning coding pairs:

\begin{definition}
 {\it $\Sigma_n$-Collection}${}^+$. Suppose that there is a class $X$ coding a quotient pair tree $\mathbb{T}_x$, in turn coding some $x \widehat{\in} (V)^+$. Let $\phi(p,q)$ be a $\Sigma_n$ formula in the language of $(V)^+$. Suppose further that for every $\mathbb{T}_y$ coding some $y \widehat{\in} (V)^+$ such that $\mathbb{T}_y E_T \mathbb{T}_x$, there is a $z \widehat{\in} (V)^+$ (represented by a class $Z$ coding a tree $\mathbb{T}_z$), such that $\phi(y, z)$. Then there is a class $A$ coding a tree $\mathbb{T}_a$ and $a \widehat{\in} (V)^+$ such that for every class $B$ coding a tree $\mathbb{T}_b E_T \mathbb{T}_x$ and $b \widehat{\in} x \widehat{\in} (V)^+$, there is a class $C$ coding a tree $\mathbb{T}_c$ and $c \widehat{\in} (V)^+$ such that $\mathbb{T}_c E_T \mathbb{T}_a$, $c \widehat{\in} a$, and $\phi(b ,c)$.
 \end{definition}

By Theorem \ref{maintheorem} we can build  versions of the $L$-hierarchy in $(V)^+$. Moreover, since Global Choice holds in the class theory, $(V)^+$ contains a class well-order $<_V$ of $V$. Let $L^{(V)^+}(V, <_V)$ then be the substructure of $(V)^+$ obtained by constructing $(V)^+$'s version of the $L$-hierarchy over $(V, <_R)$ through every meta-ordinal $\vartheta$. Then $L^{(V)^+}(V, <_V)$ validates $\Delta_n$-Separation${}^+$, since any instance of Separation for $L^{(V)^+}(V, <_V)$ translates into an instance of Separation for $(V)^+$ (this is a more specific instance of the fact that $L^\mathfrak{M}$ satisfies Separation whenever $\mathfrak{M}$ is a transitive model of Separation). So, pick the least $\varsigma$ such that $L_\varsigma(V, <_V)$ satisfies $\Delta_n$-Separation${}^+$. We claim that $L_\varsigma(V, <_V)$ also satisfies $\Sigma_n$-Collection${}^+$.

 Assume for contradiction that $\Sigma_n$-Collection fails in $L_\varsigma(V, <_V)$. Working from the perspective of $L_\varsigma (V, <_V)$, we then have a $\Sigma_n$-definable function $\phi(x,y)$ from $Ord(V)$ unbounded in $L_\varsigma (V, <_V)$. We then note that $V$ can be well-ordered in 
 \sloppy $L_\varsigma (V, <_V)$
 and therefore every element of $L_\varsigma (V, <_V)$ can be mapped into $Ord(V)$ (i.e. $Ord(V)$ is the largest cardinal of $L_\varsigma(V)$) since $\varsigma$ was chosen to be least. We can then extend $\phi(x,y)$ to a $\Sigma_n$-definable bijection between $Ord(V)$ and $L_\varsigma (V, <_V)$, to obtain a $\Sigma_n$-definable well-order $<_R$ on the subsets of $Ord(V)$ in $L_\varsigma (V, <_V)$. For $i \widehat{\in} I$ indexing $<_R$ on subsets $x_0, x_1, ..., x_i,...$ of $Ord(V)$ we then diagonalise to produce the following $X \subseteq Ord(V)$:

$$X= \{ \alpha \in On | \neg \alpha \widehat{\in} x_\alpha \}$$

This $X$ is $\Sigma_n$-definable over $Ord(V)$ by the properties of $<_R$ and the fact that $Ord(V) \widehat{\in} L_\varsigma (V, <_V)$ but cannot be in $L_\varsigma (V, <_V)$. Since $Ord(V) \widehat{\in} L_\varsigma (V, <_V)$, this violates $\Delta_n$-Separation${}^+$, $\bot$.

Thus $L_\varsigma (V, <_V)$ satisfies $\Sigma_n$-Collection${}^+$. We then know that since any instance of $\Sigma_n$-Collection${}^+$ in $L_\varsigma (V)$ is expressible in $L_\varsigma(V, <_V)$, $L_\varsigma(V)$ also satisfies $\Sigma_n$-Collection${}^+$. Thus, there is some $L_\varpi(V)$ satisfying $\Delta_0$-Separation and $\Sigma_1$-Collection, with $\varpi$ the least such.

We can now note the following:

\begin{lemma}\label{HypL}
\cite{Barwise1975a}\footnote{See \cite{Barwise1975a}, p. 60, Theorem 5.9.} For any transitive $\mathfrak{M}$, $Hyp(\mathfrak{M})$ is of the form $L_\alpha(\mathfrak{M})$ for $\alpha$ the least admissible above $\mathfrak{M}$.
\end{lemma}

\noindent and thus observe that $L_\varpi(V) \widehat{=} Hyp(V)$. Moving back to the coding, this implies the existence of a class $H$ coding a tree $\mathbb{T}_{Hyp(V)}$ for $Hyp(V)$.
\end{proof}

The above machinery then provides the resources to code $Hyp(V)$ by referring to classes. One natural question concerns how much comprehension is required to code proofs in $V$-logic. This is approximately answered by the following:

\begin{theorem}
\sloppy $\Sigma^1_1$-Comprehension is sufficient to produce a code for $Hyp(V)$ and $\Delta^1_1$-Comprehension is not sufficient to produce such a code.
\end{theorem}

\begin{proof}
To see that $\Sigma^1_1$-Comprehension is sufficient to produce a code for $Hyp(V)$, consider $H = \{x: x \widehat{\in} Hyp(V, <_V) \wedge \text{``}x \text{ is } \Sigma_1\text{-definable}$ in $ Hyp(V, <_V)$ using parameters from $V \cup \{V\} \cup \{ <_V \}$ ''$\}$. This is an extensional, $\Sigma_1$-elementary submodel of $Hyp(V, <_V)$, since we can choose least witnesses to $\Delta_0$-formulas to get $\Sigma_1$-elementarity. But the transitive collapse of $H$ must be all of $Hyp(V, <_V)$ as it is admissible and contains $V$ and $<_V$ as elements (and since $Hyp(V, <_V)$ is the smallest such structure). We then obtain a code for $Hyp(V)$ as before.

To see that $\Delta^1_1$-Comprehension is not sufficient, note that $V$ together with the subclasses of $V$ which belong to $Hyp(V)$ forms a model of $\Delta^1_1$-Comprehension. This is a model of $\Delta^1_1$-Comprehension with no code for $Hyp(V)$, since there is no code for $Hyp(V)$ inside $Hyp(V)$.
\end{proof}

In sum, we have seen thus far that $Hyp(V)$ can be coded within the class theory licensed by an interpretation of proper classes that admits a small amount of impredicativity (namely: $\mathbf{NBG}$ + $\Sigma^1_1$-Comprehension + The First Coding Lemma). We now turn to the issue of coding $V$-logic within $Hyp(V)$ (as rendered in class theory).

\subsection{Coding $V$-logic in $Hyp(V)$}

We will show that if $\phi$ is a consequence of a $V$-logic theory $\mathbf{T}$, then a proof of $\phi$ appears in $Hyp(V)$ (i.e. the (code of the) least admissible structure containing $V$), completing the rendering of $V$-logic for the Universist. Much of this goes through as in \cite{Barwise1975a}, but we have used a slightly different definition of proof code, and so this remains to be checked.

 We wish to show that if there is a proof of $\phi$ in $V$-logic, then there is a proof code of $\phi$ in $Hyp(V)$. First, however, we must be precise about how we interpret the extended syntax of $V$-logic. Since we have shown already how to code $Hyp(V)$, we work directly in the language of $Hyp(V)$ and drop the class-theoretic locutions:

\begin{definition}
($\mathbf{MK}^+$) The language and proofs of $V$-logic are interpreted as follows (working in $(V)^+$):

\begin{enumerate}[(i)]
 \item Every set $x$ is named by $\langle x, 3 \rangle$ (so, for example, if $x = \omega$, then $\bar{\omega} = \langle \omega, 3 \rangle$ (to avoid the double use of names for natural numbers and the G\"{o}del coding of the connectives).
 \item $\in^V$ and $V$ name $\in$ and $\bar{V}$ respectively (remembering that we are currently working within $(V)^+$, and these appear as sets from this perspective).
 \item The relevant $\bar{W}$ (and possibly $\bar{\mathbb{P}^C}$, $\bar{G^C}$, or any other required predicates) can be represented by any object not otherwise required for the syntax of $V$-logic, so we may use $\{ V \}, \{\{V\}\}, \{\{\{V\}\}\},...$ and so on (for any Zermelo-style construction of singletons derived from $V$).\footnote{For most axioms we only need one (and at most three) extra predicates, but we make room for the use of several different outer models in case others wish to talk about relationships between {\it incompatible} extensions using the same axiom.}
 \item After a suitable G\"{o}del coding for the connectives and quantifiers has been chosen, we represent well-formed formulas of $V$-logic with sequences of symbol codes.
 \item Proofs are represented by the appropriate trees comprising codes of the relevant sentences as nodes.
\end{enumerate}
\end{definition}

We can now state the following:

 \begin{theorem}\label{hypVproof}
 ($\mathbf{MK}^+$) \cite{Barwise1975a} Suppose that there is a proof of $\phi$ in $V$-logic. Then there is a proof code of $\phi$ in $Hyp(V)$.
 \end{theorem}
 
 \begin{proof}
We credit this to Barwise, since the proof is not very difficult, but since we have a different notion of proof-code we need to make sure that $V$-logic is still representable in $Hyp(V)$. 

 We proceed by induction on the complexity of our proof $P$. Suppose that there is a proof $P$ of $\phi$ in $V$-logic. Then either:
  
  \begin{enumerate}[(a)]
  \item $P$ is one line.
  \item $P$ is more than one line.
  \end{enumerate}
  
  We deal with (a) first. Suppose that $P$ is one line. Then either (i) $\phi$ is of the form $\bar{x} \in \bar{V}$, (ii) $\phi$ is an atomic or negated atomic sentence of $\mathscr{L}_\in \cup \{\bar{x}| x \in V \}$, (iii) $\phi$ is an axiom of first-order logic in $\mathscr{L}^V_\in$, or (iv) $\phi$ is an additional axiom containing some extra predicate (such as $\bar{W}$, $\bar{\mathbb{P}^C}$, or $\bar{G^C}$).
  
  For (i), suppose $\phi$ is of the form $\bar{x} \in \bar{V}$ for $x \in V$. Then, the result is immediate: the required sentence is in $Hyp(V)$ by Pairing, and hence so is the tree coding its proof (i.e. $\{ \{\phi \} , \emptyset \}$). In the case of (ii), $\phi$ appears in $V$, and so it is immediate that $\phi$ is in $Hyp(V)$, with the relevant proof tree. For (iii), we note that all constructions of first-order axioms from simpler formulas $\psi$ and $\chi$ (that are assumed, for induction, to be in $Hyp(V)$) can be chained together through Pairing. For (iv), since we represent the various extra predicates by objects that are not pieces of syntax in other parts of $V$-logic but are in $Hyp(V)$, any axiom of the form ``$\bar{W}$ is such that $\Phi$'' is simply a finite sequence of sets already present in $Hyp(V)$ (and similarly when $\bar{\mathbb{P}^C}$ or $\bar{G^C}$ are present). Again, repeated application of Pairing ensures that $\phi$ is in $Hyp(V)$, as well as the relevant proof tree.
  
  (b) Suppose then that $P$ is more than one line. Assume for induction that all prior steps to the final inference to $\phi$ have proofs in $Hyp(V)$. Then either (i) $\phi$ is an axiom, or (ii) $\phi$ follows from $\psi$, $(\psi \rightarrow \phi)$ via modus ponens, or (iii) $\phi$ is of the form $\forall x \in \bar{a} \psi(x)$ and follows from $\psi(\bar{b})$ for all $b \in a$ by the Set-rule, or (iv) $\phi$ is of the form $\forall x \in \bar{V} \psi(x)$ and follows from $\psi(\bar{x})$ for all $x \in V$ by the $V$-rule.
  
  For each of the steps we need to construct, from the given proof trees, a new proof tree coding a proof of $\phi$. We already know that the relevant pieces of syntax exist (by part (a)) and so the challenge is simply in the construction of the trees in $Hyp(V)$. 
  
  (i) has already been dealt with in part (a). (ii) Suppose for induction that $\psi$ and $(\psi \rightarrow \phi)$ have proofs coded in $Hyp(V)$ by $\mathbb{T}_\psi = \langle T_\psi, <_\psi \rangle$ and $\mathbb{T}_{(\psi \rightarrow \phi)} = \langle T_{(\psi \rightarrow \phi)}, <_{(\psi \rightarrow \phi)} \rangle$. Since we know that $Hyp(V)$ satisfies finite iterations of Pairing, we only need to construct $T_\phi$ and $<_\phi$. We can easily construct $T_\phi = T_\psi \cup T_{(\psi \rightarrow \phi)} \cup \{ \phi \}$. Next, we define $<_\phi$ as follows:
  
  \begin{description}
  \item $x <_\phi y$ iff:
 
  \begin{enumerate}[(i)]
  \item $x <_\psi y$, or
  \item $x <_{(\psi \rightarrow \phi)} y$, or
  \item $y = \phi$.
  \end{enumerate}
  \end{description}
  
  Since we have $<_\psi$ and $<_{(\psi \rightarrow \phi)}$ already ($Hyp(V)$ is transitive), we just need to construct $\{ \langle x, \phi \rangle | x \in T_\psi \vee x \in T_{(\psi \rightarrow \phi)} \}$. We have that $\phi \in Hyp(V)$, and also for any object $y \in Hyp(V)$, $\langle y , \phi \rangle \in Hyp(V)$. We then (working with $Hyp(V)$) define the following formula:

  $$\chi(x, y) =_{df} x \in T_\psi \cup T_{(\psi \rightarrow \phi)} \wedge y= \langle x,\phi \rangle$$
  
  $\chi(x, y)$ is clearly a $\Delta_0$ formula defining a function that maps any particular $x \in T_\psi \cup T_{(\psi \rightarrow \phi)}$ to $\langle x , \phi \rangle$. We also have the following lemma:
  
  \begin{lemma}
  ($\mathbf{MK}^+$) \cite{Barwise1975a}\footnote{See \cite{Barwise1975a}, p. 17, Theorem 4.6.} {\it $\Sigma_1$-Replacement.} $Hyp(V)$ satisfies Replacement for $\Sigma_1$ formulas.
  \end{lemma}

We thus have $\{ \langle x, \phi \rangle | x \in T_\psi \wedge x \in T_{(\psi \rightarrow \phi)} \} \in Hyp(V)$ as desired. $\langle T_\phi, <_\phi \rangle$ clearly codes a proof of $\phi$ from $\psi$ and $(\psi \rightarrow \phi)$, the proof steps are inherited from the previous trees, and each proof step prior to $\phi$ is $<_\phi$-related to $\phi$. 
 
We deal with (iii) and (iv) in tandem. As the strategy is the same for both, we give only the proof of (iv).  Suppose then that $\phi$ is of the form $\forall x \in \bar{V} \psi(x)$ and follows from $\psi(\bar{b})$ for all $b \in V$ by the $V$-rule. Assume for induction that every $\psi(\bar{b})$ has a proof code 
\sloppy $\mathbb{T}_{\psi(\bar{b})} = \langle T_{\psi(\bar{b})}, <_{\psi(\bar{b})} \rangle \in Hyp(V)$. 
We then identify:

 \begin{lemma}
 ($\mathbf{MK}^+$)  \cite{Barwise1975a}\footnote{See \cite{Barwise1975a}, p. 17, Theorem 4.4.} {\it $\Sigma_1$-Collection.} $Hyp(V)$ satisfies Collection for $\Sigma_1$ formulas.
  \end{lemma}
  
Using $\Sigma_1$-Collection, we have a non-empty set $X$ which contains proofs of $\psi(\bar{b})$ for each $b$, and a non-empty set $Y$ containing the relations of each of the $\mathbb{T}_{\psi(\bar{b})}$. Separating out from $X$ and $Y$, yields a set $X'$ containing just the domains of proof trees of $V$-logic proofs in $X$, and a set $Y'$ containing just the relations of $V$-logic proof trees in $Y$. The argument for the full tree of $\mathbb{T}_\phi$ is then exactly the same as in (ii).
\end{proof}

We now are in a position where:

\begin{enumerate}[(1)]
 \item Extensions of $V$ can be coded syntactically using $V$-logic.
 \item $Hyp(V)$ can be coded in $\mathbf{MK}^+$ (in fact,  $\mathbf{NBG}$ + $\Sigma^1_1$-Comprehension + The First Coding Lemma suffices).
 \item If $\phi$ is provable in $V$-logic, then $\phi$ has a proof code in $Hyp(V)$.
\end{enumerate}

Thus consistency and the axioms we have discussed can be coded using $\mathbf{MK}^+$ class theory. We could stop here, having shown how the Hilbertian Challenge can be met for higher-order properties of $V$, so long as some impredicative class theory is accepted.\footnote{A slight wrinkle is that though the $\mathsf{IMH}$ can be formulated in $\mathbf{NBG}$+$\Sigma^1_1$-Comprehension with the First Coding Lemma, it is always false in models of this form, since the truth predicate obtained through $\Sigma^1_1$-Comprehension implies that there is a transitive model containing every real, which contradicts the $\mathsf{IMH}$ (see \cite{Friedman2006a}, p. 597, Theorem 15 here). However, restricting the $\mathsf{IMH}$ to only take into account outer models satisfying $\Sigma^1_1$-Comprehension alleviates this worry (since the proof depends on considering least outer models that violate $\Sigma^1_1$-Comprehension), and this version of the $\mathsf{IMH}$ can also be formalised in $V$-logic (since $V$-logic captures \textit{any} syntactic theory consistent with the initial structure of $V$).} However, we still need to satisfy the Methodological Constraint, and it is to this issue that we now turn.

\section{Satisfying the Methodological Constraint}

We are still left with the question of why we should regard $V$-logic as capturing the notion of `extension of $V$', rather than merely a gerrymandered syntactic coding. Resolution of this issue is linked to responding to the Methodological Constraint; if we can show that truth in $V$-logic corresponds to truth about \textit{actual} extensions of models very similar to $V$, we would go some way towards establishing that $V$-logic stands to arbitrary extensions as the use of forcing relations stands to set forcing extensions. In this section, we'll discuss strategies for satisfying the Methodological Constraint. As we will argue, we can augment our interpretation with methods for reducing the theory of $V$-logic to the countable, thereby yielding an interpretation satisfying the Methodological Constraint.

Recall the main challenge that we wanted to satisfy:

\begin{description}
\item[The Hilbertian Challenge.] Provide philosophical reasons to legitimise the use of extra-$V$ resources for formulating axioms and analysing intra-$V$ consequences.
\end{description}

The coding performs well with respect to The Hilbertian Challenge in this raw form. We have provided philosophical reasons to accept the use of $\mathbf{MK}^+$ over $V$ using the mechanisms of class theory. We then showed how to code extensions of $V$ using $V$-logic and that $Hyp(V)$ can formalise the notion of consistency in $V$-logic. Since $Hyp(V)$ can be coded using $\mathbf{MK}^+$ over $V$, we have classes over $V$ that code consistency in $V$-logic, and hence discourse about extensions of $V$. Given the cogency of the resources of $\mathbf{MK}^+$ (argued for earlier) we can thus see why talk of extensions will not lead us astray.

The Hilbertian Challenge was, however, tempered by an additional desideratum on any interpretation of extension talk:

\begin{description}
 \item[The Methodological Constraint.] In responding to the Hilbertian Challenge, do so in a way that accounts for as much as possible of our naive thinking about extensions and links them to structural features of $V$. In particular, if we wish to apply an extending construction to $V$, there should be an {\it actual} set-theoretic model, resembling $V$ as much as possible, that has an extension similar to the one we would like $V$ to have.
\end{description}

As it stands, we do {\it not} have a naive interpretation of extensions of $V$: We are interpreting model-theoretic claims about extensions of $V$ as the syntactic consistency of theories in $V$-logic. We would like to find a place for our naive thinking concerning extensions and relate this discourse to our analysis of truth in $V$. The key fact here is that $Hyp(V)$ (and hence claims about consistency in $V$-logic) is coded by a {\it single} class. As we shall argue, by reducing the theory of $Hyp(V)$ to a countable model, we yield an interpretation of extension talk satisfying the Methodological Constraint.

Recall that for parameter-free first-order truth, the countable transitive model strategy fared reasonably well, barring its failure to account for greater than first-order axioms. In responding to the Hilbertian Challenge and trying to satisfy the Methodological Constraint, we shall thus pursue the strategy of finding a countable transitive model that mirrors $V$ with respect to the theory of $V$-logic. As we shall see, this then facilitates the formulation of axioms about $V$ that make use of extension talk, whilst finding an arena for our naive thinking.

One might try to argue for the existence of such a countable transitive model informally.\footnote{Say by postulating Skolem functions of the required kind. This is a strategy advocated at one stage by Cohen, for example when he writes:

\begin{quote}
The L\"{o}wenheim-Skolem theorem allows us to pass to countable submodels of a given model. Now, the ``universe'' does not form a set and so we cannot, in $\mathbf{ZF}$, prove the existence of a countable sub-model. However, informally we can repeat the proof of the theorem. We recall that the proof merely consisted of choosing successively sets which satisfied certain properties, if such a set existed. In $\mathbf{ZF}$ we can do this process finitely often. There is no reason to believe that in the real world this process cannot be done countably many times and thus yield a countable standard model for $\mathbf{ZF}$. (\cite{Cohen1966a}, p79)
\end{quote}

One might adapt this idea to the current context by running a similar argument, but introducing a predicate $H$ for (the class-theoretic code of) $Hyp(V)$ into the language of $\mathbf{ZFC}$.} Fortunately, we can do better in the current context by availing ourselves of the following result:\footnote{We are very grateful to Kameryn Williams for discussion here.}

\begin{lemma}
(Folklore)\footnote{Presentations of both these lemmas are available in \cite{Fujimoto2012a} (esp. p. 1514) and \cite{GitmanHamkins2017a} (esp \S4).} $\mathsf{ETR}$ (in fact $\mathsf{ETR}$ for recursions of length $\omega$) implies the existence of truth predicates relative to any class.
\end{lemma}

\begin{lemma}
(Folklore) $\mathsf{ETR}$ is provable in $\mathsf{NBG} + \Sigma^1_1$-Comprehension.
\end{lemma}

In particular since $Hyp(V)$ (coded in the class theory) is a well-founded class relation, using $\Sigma^1_1$-Comprehension we can prove the existence of a truth predicate relative to this class using $\mathsf{ETR}$. Letting $H(x)$ be a predicate applying exactly when $x \in Hyp(V)$, and $T_H$ be a truth predicate for the language of $\mathscr{L}_\mathbf{NBG} \cup \{ H(x) \}$, we can then prove (using reflection) that there is a $V_\alpha$ that is elementary in $V$ for truth in $\mathscr{L}_\mathbf{NBG} \cup \{ H(x) \}$ (in fact, $\mathsf{ETR}$ implies that $V$ is a tower of such universes\footnote{See here \cite{GitmanHamkins2017a}, p. 129.}). Using the L\"{o}wenheim-Skolem and Mostowski Collapse Theorems over this model (i.e. $(V_\alpha, \in, Hyp(V_\alpha) )$) yields a countable transitive model $\mathfrak{V}^*= (V^*, \in, Hyp(V^*))$ elementary to $V$ for truth in $\mathscr{L}_\mathbf{NBG} \cup \{ H(x) \}$, and hence the corresponding $\mathfrak{V}^*$-logic (formalisable in $Hyp(\mathfrak{V}^*)$) agrees with $V$-logic. The difference being here, of course, that since $\mathfrak{V}^*$ is countable, the Barwise completeness theorem\footnote{See \cite{Barwise1975a}, p. 99, Theorem 5.5.} holds, and hence there is a model \textit{very similar to} $V$ that is \textit{really} extended when we consider different kinds of construction. Put in concrete terms, we have:

\begin{fact}
($\mathbf{MK}^+$) Let $\phi$ be a sentence of $V$-logic with no constant symbols apart from $\bar{V}$. Then the following are equivalent:

\begin{enumerate}[(1.)]
\item $\phi$ is consistent in $V$-logic.
\item $\phi$ is consistent in $\mathfrak{V}^*$-logic.
\item $\mathfrak{V}^*$ has an outer model with $\phi$ true.
\end{enumerate}
\end{fact}

Thus when we want to know if it is consistent to have an `outer model' of $V$ satisfying some first-order property mentioning $V$ as a predicate with parameters definable in $V$, we can look to $\mathfrak{V}^*$ (where extensions are readily available). Of course one does not consider ``extensions of $V$'' as having anything other than syntactic meaning, but one can find an exceptionally close simulacrum of this reasoning in the countable and explain how its truth and structure is related to $V$ (via the correspondence between $V$ and $\mathfrak{V}^*$ with respect to first-order and admissible truth). Thus, by linking $V$-logic to the countable, we augment our answer to the Hilbertian Challenge with a response to the Methodological Constraint, explaining the existence of particular models that closely resemble the naive reasoning with extensions and use of the term `$V$' in set-theoretic discourse.

\subsubsection*{Omniscience}

In the previous subsection, we examined the possibility of reducing the theory of $Hyp(V)$ to a countable transitive model $\mathfrak{V}^*$. In this subsection, we examine one further development, namely that if $V$ satisfies certain greater than first-order properties then the theory of $V$'s outer models becomes {\it first-order definable} in $V$.
 
Are there conditions that $V$ might satisfy allowing first-order access to the theory of its outer models? \cite{FriedmanHonzik2016b}, building on work of \cite{StanleyUa} examine exactly this. First, however, we note a limitation to any attempt of this kind:

  \begin{theorem}
  ($\mathbf{ZFC}$) \cite{StanleyUa}\footnote{This theorem is a quick consequence of the Fixed Point Lemma. See \cite{FriedmanHonzik2016b} for details.} No first-order hypotheses on $V$ suffice to give it access to the theory of its outer models. More precisely, assume that:
   \begin{enumerate}[(i)]
    \item $\mathbf{ZFC}^*$ is a recursive first-order theory extending $\mathbf{ZFC}$, and that $\mathbf{ZFC}^*$ has countable standard transitive models.
    \item $\mathsf{good}(x)$ is a parameter-free formula such that if $\mathfrak{M}$ is a countable standard transitive model of $\mathbf{ZFC}^*$, $\mathbf{T} \in \mathfrak{M}$ is a set of axioms and $\mathfrak{M} \models \mathsf{good}[\mathbf{T}]$, then $\mathfrak{M}$ has an outer model that satisfies $\mathbf{T}$. 
    \item $\mathsf{bad}(x)$ is a parameter-free formula such that if $\mathfrak{M}$ is a countable standard transitive model of $\mathbf{ZFC}^*$, $\mathbf{T} \in \mathfrak{M}$ is a set of axioms and $\mathfrak{M} \models \mathsf{bad}[\mathbf{T}]$, then $\mathfrak{M}$ does not have an outer model that satisfies $\mathbf{T}$.
   \end{enumerate}

   Then there exists a recursive $T \supseteq \mathbf{ZFC}^*$ such that neither $\mathsf{good}[T]$ nor $\mathsf{bad}[T]$ holds in any countable standard transitive model of $\mathbf{ZFC}^*$.
  \end{theorem}

 The theorem shows that we will \textit{never} be able to give {\it purely} first-order conditions that give $V$ access to its outer models. However, there is still the possibility that {\it second-order} hypotheses might confer {\it first-order} definability on the theory of $V$'s outer models. Indeed this is so (given large cardinals):

  \begin{theorem}
  ($\mathbf{NBG}$) \cite{StanleyUa} (restated and proved in \cite{FriedmanHonzik2016b}) Suppose that $\mathfrak{M}$ is a transitive model of $\mathbf{ZFC}$ (of which $V$ is one). Suppose that in $\mathfrak{M}$ there is a proper class of measurable cardinals, and this class is $Hyp(\mathfrak{M})$-stationary, i.e. $Ord(\mathfrak{M})$ is regular with respect to $Hyp(\mathfrak{M})$-definable functions and the class of all measurable cardinals in $\mathfrak{M}$ intersects every club in $Ord(\mathfrak{M})$ which is $Hyp(\mathfrak{M})$-definable. Then satisfaction in outer models of $\mathfrak{M}$ is first-order definable over $\mathfrak{M}$.
  \end{theorem}

 What does this theorem show? Namely that if there are classes with certain slightly greater than first-order properties (the $Hyp(V)$-stationarity of the measurables is not first-order definable), then there is a first-order formula that captures satisfaction in outer models of $V$. Indeed, the consistency strength of outer-model satisfaction being first-order definable within a model has currently been found consistent relative to the existence of a single inaccessible cardinal.\footnote{See \cite{FriedmanHonzik2016b} for discussion.} This provides a different perspective on outer model satisfaction, and again shows the relevance of $Hyp(V)$: As long as $V$ satisfies the existence of the relevant measurables and their $Hyp(V)$-stationarity, then if there is a countable transitive model elementarily equivalent to $V$ for $\mathbf{ZFC}$ augmented with a predicate for $Hyp(V)$, then we can behave exactly as if outer models of $V$ actually existed. In this case, our original $\mathfrak{V}$, equivalent to $V$ for first-order truth, would suffice for interpreting extension talk concerning $V$.
 
 In sum, as long as we have a countable transitive model that sufficiently mirrors $V$, we then have the Methodological Constraint satisfied. Whenever we reason naively about extensions, we can interpret this as concerned with the relevant countable transitive model (i.e. $\mathfrak{V}$, or $\mathfrak{V}^*$). Here, extensions are uncontroversially available, and so we may perfectly well reason naively and combinatorially about {\it them}, safe in the knowledge that any facts so discovered will be mirrored by $V$. 

\section{Open Questions}

Before we conclude, we make a few final remarks concerning salient open questions. The first concerns what is possible using the mechanisms of $V$-logic. In this paper, we have only examined {\it width} extensions, and made some remarks about height or non-well-founded extensions. However, there is some extension talk which seems to require that $V$ appear non-standard from a different perspective. \cite{Hamkins2012a} and \cite{GitmanHamkins2010a} for example consider the following as part of their Multiverse Axioms:

 \begin{axiom}
  {\it Well-foundedness Mirage.} Every universe (including $V$ situated inside a multiverse) is non-well-founded from the perspective of another universe.
 \end{axiom}

One issue here is that $V$-logic {\it explicitly} keeps $V$ standard. While we can interpret non-well-founded universes using $V$-logic, as it stands $V$ will always appear in them as a standard, well-founded part. We then have the following question:

 \begin{question}
  How much talk concerning other kinds of extensions of $V$ can be interpreted using $V$-logic (or similar constructions)?
 \end{question}

A further question concerns the role of class theory. Throughout this paper we have assumed the use of some impredicative class theory. As we indicated, there are many options for a Universist to interpret impredicative class theory in a philosophically motivated fashion. Despite this, there will no doubt be some who find the use of impredicative class theory uncomfortable. We might then, in this spirit, ask the following question:

 \begin{question}
  How much $V$-logic can we capture without the use of any class theory?
 \end{question}

 This in turn raises the following problem:
 
 \begin{question}
  How much talk concerning {\it arbitrary} extensions can be captured without the use of impredicative classes?
 \end{question}

Finally, we recall a previously mentioned open :
 
 \vspace{0.2cm}
 
 \noindent {\bf Question \ref{IMH1}.} Could there be a countable transitive model $\mathfrak{M}=(M, \in, C^\mathfrak{M})$ of $\mathbf{NBG}$, such that $\mathfrak{M}$ has a countable (from the perspective of $\mathfrak{M}$) transitive submodel 
 \sloppy $\mathfrak{M}'=(M', \in, C^{\mathfrak{M}'})$, 
 also a model of $\mathbf{NBG}$, with $\mathfrak{M}$ and $\mathfrak{M}'$ agreeing on parameter-free first-order truth in $\mathbf{ZFC}$ but disagreeing on the $\mathsf{IMH}$?
 
\section*{Conclusions}

 We have seen that the Universist has some reason to want to use extension talk in  formulating axioms about $V$. An expansion of logical resources, combined with a degree of impredicative comprehension, facilitates an interpretation of extension talk that finds a place for our naive reasoning concerning extensions.

We make some final remarks concerning the state of the dialectic. We wish to be conservative about the philosophical implications of our results. What we have established is that those Universists who accept that extension talk concerning $V$ is worth scrutiny and who think that the Methodological Constraint is important in answering the Hilbertian Challenge can encode a substantial amount of talk concerning extensions of $V$ using sets from $V$ and/or classes over $V$, so long as they regard some impredicative class theory as legitimate and think that we can extend our logical resources using $V$-logic. There are an awful lot of moving parts involved in setting up this dialectic and we should be mindful that there are different upshots one might take. Obviously, one might think that the paper shows that extension talk can be used freely by the Universist. However,  one might instead take our results to show that there is a problem with talking about extended languages or the use of $\mathbf{MK}^+$. Or one might think that the Methodological Constraint is too strict, or not strict enough. Further still, one might take the sheer amount of talk of extensions of $V$ that can be coded within the Universist's framework to be evidence of the falsity of their position. We have not taken a stand on any of these issues here, but have shown what can be accomplished given the acceptance of certain positions and resources. Nonetheless, the fact remains that with a smidgeon of extra expressive resources, the Universist can utilise far more mathematics than previously thought. The doors are thus open to new and intriguing philosophical and mathematical discussions.


\begin{thebibliography}{}

\bibitem[Antos, 2015]{Antos2015a}
Antos, C. (2015).
\newblock {\em Foundations of Higher-Order Forcing}.
\newblock PhD thesis, Universit\"{a}t Wien.

\bibitem[Antos, 2018]{Antos2018a}
Antos, C. (2018).
\newblock {\em Class Forcing in Class Theory}, pages 1--16.
\newblock Springer International Publishing, Cham.

\bibitem[Antos and Friedman, 2017]{AntosFriedman2017a}
Antos, C. and Friedman, S.-D. (2017).
\newblock Hyperclass forcing in {M}orse-{K}elley class theory.
\newblock {\em The Journal of Symbolic Logic}, 82(2):549–575.

\bibitem[Antos et~al., 2015]{AntosFriedmanHonzikTernullo2015a}
Antos, C., Friedman, S.-D., Honzik, R., and Ternullo, C. (2015).
\newblock Multiverse conceptions in set theory.
\newblock {\em Synthese}, 192(8):2463--2488.

\bibitem[Arrigoni and Friedman, 2013]{ArrigoniFriedman2013a}
Arrigoni, T. and Friedman, S.-D. (2013).
\newblock The {H}yperuniverse {P}rogram.
\newblock {\em Bulletin of Symbolic Logic}, 19:77--96.

\bibitem[Bagaria, 1997]{Bagaria1997a}
Bagaria, J. (1997).
\newblock A characterization of {M}artin's {A}xiom in terms of absoluteness.
\newblock {\em The Journal of Symbolic Logic}, 62(2):366--372.

\bibitem[Bagaria, 2005]{Bagaria2005a}
Bagaria, J. (2005).
\newblock Natural axioms of set theory and the continuum problem.
\newblock In {\em Proceedings of the 12th International Congress of Logic,
  Methodology, and Philosophy of Science}, pages 43--64. King’s College
  London Publications.

\bibitem[Bagaria, 2008]{Bagaria2008a}
Bagaria, J. (2008).
\newblock Set theory.
\newblock In {\em The Princeton Companion to Mathematics}, pages 302--321.
  Princeton University Press.

\bibitem[Barton, 2019]{Barton2019b}
Barton, N. (2019).
\newblock Forcing and the universe of sets: Must we lose insight?
\newblock {\em Journal of Philosophical Logic}.

\bibitem[Barton, S]{BartonSa}
Barton, N. (S).
\newblock Are large cardinal axioms restrictive?
\newblock Manuscript under review. Pre-Print:
  \url{https://philpapers.org/archive/BARALC-5.pdf}.

\bibitem[Barton and Friedman, 2017]{BartonFriedman2017a}
Barton, N. and Friedman, S.-D. (2017).
\newblock Maximality and ontology: how axiom content varies across
  philosophical frameworks.
\newblock {\em Synthese}, 197(2):623--649.

\bibitem[Barwise, 1975]{Barwise1975a}
Barwise, J. (1975).
\newblock {\em Admissible Sets and Structures: An approach to definability
  theory}.
\newblock Springer.

\bibitem[Boolos, 1984]{Boolos1984a}
Boolos, G. (1984).
\newblock To be is to be a value of a variable (or to be some values of some
  variables).
\newblock {\em The Journal of Philosophy}, 81(8):430--449.

\bibitem[Bukovsk\'{y}, 1973]{Bukovsky1973a}
Bukovsk\'{y}, L. (1973).
\newblock Characterization of generic extensions of models of set theory.
\newblock {\em Fundamenta Mathematicae}, 83(1):35--46.

\bibitem[Caicedo et~al., 2017]{CaicedoCummingsKoellnerLarson2017a}
Caicedo, A.~E., Cummings, J., Koellner, P., and Larson, P.~B., editors (2017).
\newblock {\em Foundations of Mathematics: Logic at Harvard Essays in Honor of
  W. Hugh Woodin's 60th Birthday}, volume 690 of {\em Contemporary
  Mathematics}.
\newblock American Mathematical Society.

\bibitem[Cohen, 1966]{Cohen1966a}
Cohen, P. (1966).
\newblock {\em Set Theory and The Continuum Hypothesis}.
\newblock W.A. Benjamin, Inc.

\bibitem[Cummings, 2010]{Cummings2010a}
Cummings, J. (2010).
\newblock {\em Iterated Forcing and Elementary Embeddings}, pages 775--883.
\newblock Springer Netherlands, Dordrecht.

\bibitem[Detlefsen, 1986]{Detlefsen1986a}
Detlefsen, M. (1986).
\newblock {\em {H}ilbert's Program}.
\newblock Dordrecht: Reidel.

\bibitem[Detlefsen, 1990]{Detlefsen1990a}
Detlefsen, M. (1990).
\newblock On an alleged refutation of {H}ilbert's program using {G}\"{o}del's
  first incompleteness theorem.
\newblock {\em Journal of Philosophical Logic}, 19:343--377.

\bibitem[Detlefsen, 2001]{Detlefsen2001a}
Detlefsen, M. (2001).
\newblock What does {G}\"{o}del's {S}econd {T}heorem say?
\newblock {\em Philosophia Mathematica}, 9:37--71.

\bibitem[Enderton, 1972]{Enderton1972a}
Enderton, H. (1972).
\newblock {\em A Mathematical Introduction to Logic}.
\newblock Harcourt/Academic Press.

\bibitem[Ewald, 1996]{Ewald1996a}
Ewald, W.~B., editor (1996).
\newblock {\em From Kant to Hilbert. A Source Book in the Foundations of
  Mathematics}, volume~I.
\newblock Oxford University Press.

\bibitem[Friedman, 2000]{Friedman2000a}
Friedman, S.-D. (2000).
\newblock {\em Fine Structure and Class Forcing}.
\newblock de Gruyter.
\newblock de Gruyter Series in Logic and its Applications, Vol. 3.

\bibitem[Friedman, 2006]{Friedman2006a}
Friedman, S.-D. (2006).
\newblock Internal consistency and the inner model hypothesis.
\newblock {\em Bulletin of Symbolic Logic}, 12(4):591--600.

\bibitem[Friedman, 2016]{Friedman2016a}
Friedman, S.-D. (2016).
\newblock Evidence for set-theoretic truth and the hyperuniverse programme.
\newblock {\em IfCoLog J. Log. Appl.}, 4(“Proof, Truth, Computation”
  3):517--555.

\bibitem[Friedman et~al., F]{FriedmanFuchinoSakaiFa}
Friedman, S.-D., Fuchino, S., and Sakai, H. (F).
\newblock On the set-generic multiverse.
\newblock {\em Proceedings of the Singapore Programme on Sets and
  Computations}.
\newblock Forthcoming.

\bibitem[Friedman and Honzik, 2016a]{FriedmanHonzik2016b}
Friedman, S.-D. and Honzik, R. (2016a).
\newblock Definability of satisfaction in outer models.
\newblock {\em Journal of Symbolic Logic}, 81(3):1047--1068.

\bibitem[Friedman and Honzik, 2016b]{FriedmanHonzik2016a}
Friedman, S.-D. and Honzik, R. (2016b).
\newblock On strong forms of reflection in set theory.
\newblock {\em Mathematical Logic Quarterly}, 62(1-2):52--58.

\bibitem[Friedman et~al., 2008]{FriedmanWelchWoodin2008a}
Friedman, S.-D., Welch, P., and Woodin, W.~H. (2008).
\newblock On the consistency strength of the inner model hypothesis.
\newblock {\em The Journal of Symbolic Logic}, 73(2):391--400.

\bibitem[Fuchs et~al., 2015]{FuchsHamkinsReitz2015a}
Fuchs, G., Hamkins, J.~D., and Reitz, J. (2015).
\newblock Set-theoretic geology.
\newblock {\em Annals of Pure and Applied Logic}, 166(4):464--501.

\bibitem[Fujimoto, 2012]{Fujimoto2012a}
Fujimoto, K. (2012).
\newblock Classes and truths in set theory.
\newblock {\em Annals of Pure and Applied Logic}, 163:1484--1423.

\bibitem[Gitman and Hamkins, 2010]{GitmanHamkins2010a}
Gitman, V. and Hamkins, J.~D. (2010).
\newblock A natural model of the multiverse axioms.
\newblock {\em Notre Dame Journal of Formal Logic}, 51(4):475--484.

\bibitem[Gitman and Hamkins, 2017]{GitmanHamkins2017a}
Gitman, V. and Hamkins, J.~D. (2017).
\newblock Open determinacy for class games.
\newblock In {\em \cite{CaicedoCummingsKoellnerLarson2017a}}, pages 121--144.
  American Mathematical Society.

\bibitem[Hale, 2013]{Hale2013a}
Hale, B. (2013).
\newblock {\em Necessary Beings: An Essay on Ontology, Modality, and the
  Relations Between Them}.
\newblock Oxford University Press.

\bibitem[Hamkins, 2012]{Hamkins2012a}
Hamkins, J.~D. (2012).
\newblock The set-theoretic multiverse.
\newblock {\em The Review of Symbolic Logic}, 5(3):416--449.

\bibitem[Hellman, 1989]{Hellman1989a}
Hellman, G. (1989).
\newblock {\em Mathematics Without Numbers}.
\newblock Oxford University Press.

\bibitem[Isaacson, 2011]{Isaacson2011a}
Isaacson, D. (2011).
\newblock The reality of mathematics and the case of set theory.
\newblock In Noviak, Z. and Simonyi, A., editors, {\em Truth, Reference, and
  Realism}, pages 1--75. Central European University Press.

\bibitem[Koellner, 2013]{Koellner2013a}
Koellner, P. (2013).
\newblock Hamkins on the multiverse.
\newblock In Koellner, P., editor, {\em Exploring the Frontiers of
  Incompleteness}.

\bibitem[Kunen, 2013]{Kunen2013a}
Kunen, K. (2013).
\newblock {\em Set Theory}.
\newblock College Publications.

\bibitem[Laver, 2007]{Laver2007a}
Laver, R. (2007).
\newblock Certain very large cardinals are not created in small forcing
  extensions.
\newblock {\em Annals of Pure and Applied Logic}, 149(1–3):1 -- 6.

\bibitem[Linnebo, 2006]{Linnebo2006b}
Linnebo, {\O}. (2006).
\newblock Sets, properties, and unrestricted quantification.
\newblock In {\em Absolute Generality}. Oxford University Press.

\bibitem[Meadows, 2015]{Meadows2015a}
Meadows, T. (2015).
\newblock Naive infinitism.
\newblock {\em Notre Dame Journal of Formal Logic}, 56(1):191--212.

\bibitem[Rayo and Uzquiano, 2006]{RayoUzquiano2006a}
Rayo, A. and Uzquiano, G. (2006).
\newblock Introduction to \emph{Absolute Generality}.
\newblock In Rayo, A. and Uzquiano, G., editors, {\em Absolute Generality},
  pages 1--19. Oxford University Press.

\bibitem[Rumfitt, 2015]{Rumfitt2015a}
Rumfitt, I. (2015).
\newblock {\em The Boundary Stones of Thought: An Essay in the Philosophy of
  Logic}.
\newblock Oxford University Press.

\bibitem[Stanley, 2008]{StanleyUa}
Stanley, M.~C. (2008).
\newblock Outer model satisfiability.
\newblock Unpublished manuscript.

\bibitem[Steel, 2014]{Steel2014a}
Steel, J. (2014).
\newblock G\"{o}del's program.
\newblock In Kennedy, J., editor, {\em Interpreting G\"{o}del}. Cambridge
  University Press.

\bibitem[Usuba, 2017]{Usuba2017a}
Usuba, T. (2017).
\newblock The downward directed grounds hypothesis and very large cardinals.
\newblock {\em Journal of Mathematical Logic}, 17(2):1750009.

\bibitem[Uzquiano, 2003]{Uzquiano2003a}
Uzquiano, G. (2003).
\newblock Plural quantification and classes.
\newblock {\em Philosophia Mathematica}, 11(1):67--81.

\bibitem[Welch and Horsten, 2016]{HorstenWelch2016a}
Welch, P. and Horsten, L. (2016).
\newblock Reflecting on absolute infinity.
\newblock {\em Journal of Philosophy}, 113(2):89--111.

\bibitem[Zermelo, 1930]{Zermelo1930a}
Zermelo, E. (1930).
\newblock On boundary numbers and domains of sets.
\newblock In {\em \cite{Ewald1996a}}, volume~2, pages 1208--1233. Oxford
  University Press.

\end{thebibliography}
\end{document}